\documentclass[1 leqno,11pt,psfig]{amsart}
\usepackage{amssymb, amsmath,amsmath,latexsym,amssymb,amsfonts,amsbsy, amsthm,mathtools,graphicx,CJKutf8,CJKnumb,CJKulem,color}
\usepackage{graphicx,epsfig}
\usepackage{graphicx}
\usepackage{amssymb}
\usepackage{graphicx}
\usepackage{graphicx,epsfig}
\usepackage{amsmath}
\usepackage{mathrsfs}
\usepackage{amssymb}

\usepackage{amssymb,mathrsfs,amsfonts,amsmath,amsbsy,tikz}\usepackage{fontenc}\usepackage{textcomp}

\numberwithin{equation}{section}

\allowdisplaybreaks

\let\S=\Sigma

\def\R{\mathbf R}

\def\C{\mathbb{C}}
\def\R{\mathbb{R}}

\def\A{\mathbf{A}}
\def\B{\mathbf{B}}

\newcommand{\beq}{\begin{equation}}
\newcommand{\eeq}{\end{equation}}
\newcommand{\ben}{\begin{eqnarray}}
\newcommand{\een}{\end{eqnarray}}
\newcommand{\beno}{\begin{eqnarray*}}
\newcommand{\eeno}{\end{eqnarray*}}

\let\ld=\lambda

\let\af=\alpha
\let\bt=\beta

\let\Ld=\Lambda

\newcommand{\ba}{\begin{array}}
\newcommand{\ea}{\end{array}}
\newcommand{\be}{\begin{equation}}
\newcommand{\ee}{\end{equation}}
\newcommand{\ban}{\begin{eqnarray*}}
\newcommand{\ean}{\end{eqnarray*}}

\numberwithin{equation}{section}

\newtheorem{theorem}{Theorem}[section]

\newtheorem{lemma}[theorem]{Lemma}
\newtheorem{proposition}[theorem]{Proposition}
\newtheorem{corol}[theorem]{Corollary}
\newtheorem{remark}[theorem]{Remark}

\begin{document}
\begin{CJK*}{UTF8}{gkai}
\title[Discrete Sturm-Liouville problems] {Discrete Sturm-Liouville problems: singularity of the $n$-th eigenvalue  with application to Atkinson type}

\author{Guojing Ren}
\address{School of Mathematics and Quantitative Economics,
 Shandong  University of Finance and Economics, Jinan, Shandong 250014, P. R.
China}
\email{gjren@sdufe.edu.cn}

\author{Hao Zhu}
\address{Chern Institute of Mathematics, Nankai University, Tianjin 300071, P. R. China}

\email{haozhu@nankai.edu.cn}

\maketitle

\begin{abstract}
 In this paper, we  characterize  singularity of the $n$-th eigenvalue of self-adjoint discrete   Sturm-Liouville problems in any dimension.
For a fixed Sturm-Liouville equation, we  completely characterize   singularity of the $n$-th eigenvalue. For a fixed boundary condition, unlike in the continuous case,  the $n$-th eigenvalue exhibits jump phenomena  and we describe the singularity   under a non-degenerate assumption.
 Compared with the continuous   case in \cite{Hu-Liu-Wu-Zhu2018,Kong2}, the singular set here  is involved heavily  with  coefficients of the Sturm-Liouville equations. This, along with  arbitrariness of the dimension, causes difficulty when dividing areas in layers of the considered space such that the $n$-th eigenvalue has the same singularity in any given area. We study the singularity by partitioning and analyzing the local coordinate systems, and provide a Hermitian matrix  which can determine the areas' division.
  To  prove the asymptotic behavior of the $n$-th eigenvalue, we generalize the
   method developed  in \cite{Zhu2}  to any dimension. Finally, by transforming  the Sturm-Liouville problem of Atkinson type in
   any  dimension to a discrete one, we can not only determine the number of eigenvalues, but also apply our approach above to obtain the complete characterization of  singularity of the $n$-th eigenvalue for the Atkinson type.
\end{abstract}

{\bf \it Keywords}:\ Discrete  Sturm-Liouville problem;  the $n$-th eigenvalue; singularity;  boundary condition;  Atkinson type.

{2010 {\bf \it Mathematics Subject Classification}}: 39A12, 34B24, 39A70.

\section{Introduction}

Discrete Sturm-Liouville problems come from several physical models, including the vibrating string and random walk with discrete time process
\cite{Atkinson1,Jirari}.
We briefly introduce these two models. Suppose that a weightless string bears  $l$ particles with masses $m_1,\cdots, m_{l}$,  and the horizontal distance between $m_i$ and $m_{i+1}$ is $1/c_{i}$, $1\leq i\leq l-1$. Moreover, the string extends to length $1/c_{l}$ beyond $m_{l}$ and $1/c_{0}$ beyond $m_1$. Let $s_i$, $1\leq i\leq l$, be the displacement of the particle $m_i$ at a fixed time. Both ends are pinned down (i.e., $s_0=s_{l+1}=0$).
Since the particle $m_i$ does not move horizontally, we may assume that the horizontal component of the tension at  $m_i$ is both unit from the left and right, respectively.
 Then the restoring forces, induced by the  vertical component of the tension from the left and right, are $c_{i-1}(s_i-s_{i-1})$ and $c_i(s_i-s_{i+1})$, respectively.
\begin{center}
 \begin{tikzpicture}[scale=0.56]
\draw [-](-2, 0)--(2, 0)node[right]{};
\path [-](-2,0) edge [-,dotted](-2,2)  [line width=0.7pt];
\path [-](0,0) edge [-,dotted](0,5)  [line width=0.7pt];
\path [-](2,0) edge [-,dotted](2,3)  [line width=0.7pt];
\draw [-](-2, 2)--(0, 5)node[right]{};
\draw [-](0, 5)--(2, 3)node[right]{};
 \node (a) at (-2.5,2.5) {\small$m_{i-1}$};
  \node (a) at (0,5.5) {\small$m_{i}$};
   \node (a) at (2.5,3.5) {\small$m_{i+1}$};
   \node (a) at (-1,-0.5) {\small$1/c_{i-1}$};
   \node (a) at (1,-0.5) {\small$1/c_{i}$};
    \node (a) at (-1.4,1) {\small$s_{i-1}$};
    \node (a) at (0.3,1) {\small$s_{i}$};
     \node (a) at (2.6,1) {\small$s_{i+1}$};
 \end{tikzpicture}
\end{center}\vspace{-0.1cm}
 \begin{center}\vspace{-0.1cm}
   {\small {\bf Figure 1.} }
  \end{center}\vspace{-0.2cm}
Therefore, by Newton's second law,
\begin{align}\label{vibrating string1}-m_i{d^2\over dt^2} s_i=c_{i-1}(s_i-s_{i-1})+c_i(s_i-s_{i+1})=-\nabla(c_i\Delta s_i),\;1\leq i\leq l,\end{align}
where $\Delta s_i=s_{i+1}-s_{i}$ and $\nabla s_i=s_i-s_{i-1}$.
 Taking $s_i=y_i\cos(\omega t),$ where  $y_i$ is the amplitude  of $m_i$,
  we obtain from \eqref{vibrating string1} that
  \begin{align*}
  -\nabla(c_i\Delta y_i)=\lambda m_i y_i, \;1\leq i\leq l,
  \end{align*}
  where $\lambda=\omega^2$. Since the boundary condition corresponds to the assumption that both ends are pinned down,  this system becomes a self-adjoint discrete Sturm-Liouville problem.

Another model of the discrete Sturm-Liouville equation is random walking  with discrete time process from probability theory.

\begin{center}
 \begin{tikzpicture}[scale=0.8]
\draw [-](-8, 0)--(8, 0)node[right]{};
\draw [-](-8, -0.1)--(-8, 0.1)node[right]{};
\draw [-](-5, -0.1)--(-5, 0.1)node[right]{};
\draw [-](0, -0.1)--(0, 0.1)node[right]{};
\draw [-](8, -0.1)--(8, 0.1)node[right]{};
\draw [-](5, -0.1)--(5, 0.1)node[right]{};
\node (a) at (-8.5,0.5) {\small$\longleftarrow$};
\node (a) at (-8.5,1) {\small$\beta_1$};
\node (a) at (-7.5,0.5) {\small$\longrightarrow$};
\node (a) at (-7.5,1) {\small$\alpha_1$};
\node (a) at (-5.5,0.5) {\small$\longleftarrow$};
\node (a) at (-5.5,1) {\small$\beta_2$};
\node (a) at (-4.5,0.5) {\small$\longrightarrow$};
\node (a) at (-4.5,1) {\small$\alpha_2$};
\node (a) at (-2.5,0.5) {\small$\cdots\cdots$};
\node (a) at (-0.5,0.5) {\small$\longleftarrow$};
\node (a) at (-0.5,1) {\small$\beta_i$};
\node (a) at (0.5,0.5) {\small$\longrightarrow$};
\node (a) at (0.5,1) {\small$\alpha_i$};
\node (a) at (8.5,0.5) {\small$\longrightarrow$};
\node (a) at (8.5,1) {\small$\alpha_{l}$};
\node (a) at (7.5,0.5) {\small$\longleftarrow$};
\node (a) at (7.5,1) {\small$\beta_{l}$};
\node (a) at (5.5,0.5) {\small$\longrightarrow$};
\node (a) at (5.5,1) {\small$\alpha_{l-1}$};
\node (a) at (4.5,0.5) {\small$\longleftarrow$};
\node (a) at (4.5,1) {\small$\beta_{l-1}$};
\node (a) at (2.5,0.5) {\small$\cdots\cdots$};
\node (a) at (-8,-0.5) {\small$1$};
\node (a) at (-5,-0.5) {\small$2$};
\node (a) at (0,-0.5) {\small$i$};
\node (a) at (5,-0.5) {\small$l-1$};
\node (a) at (8,-0.5) {\small$l$};
 \end{tikzpicture}
\end{center}\vspace{-0.2cm}
 \begin{center}\vspace{-0.2cm}
   {\small {\bf Figure 2.} }
  \end{center}
Let a particle be in one of the $l$ positions $1,\cdots,l$ at $t=0$. Suppose that the particle is in position $i$ at $t=t_0$.
The rule of this random walking  is that the particle will move to $i+1$ at $t=t_0+1$ with a probability $\alpha_i$, move to $i-1$ at $t=t_0+1$ with a probability $\beta_i$, and stay  in position $i$ with a probability $1-\alpha_i-\beta_i$.
Moreover, if the particle moves to the left of position $1$, or to the right of position $l$, it is considered permanently lost.
So it is reasonable to set $\alpha_{0}=0$ and $\beta_{l+1}=0$.
Define $p_{rs}(j)$ as the probability of the particle being in position $s$ at $t=j$ and starting in position $r$ at $t=0$. Then we have
$p_{rs}(0)=\delta_{rs}$, and
\begin{align}\label{random-walk}
p_{rs}(j+1)=\alpha_{s-1}p_{r,s-1}(j)+\beta_{s+1}p_{r,s+1}(j)+(1-\alpha_s-\beta_s)p_{rs}(j), \; j\geq0,
\end{align}
where $\delta_{rs}=1$ if $r=s$, and $\delta_{rs}=0$ if $r\neq s$. Let $P(j)=(p_{rs}(j))_{1\leq r,s\leq l}$, $j\geq 0$, and
\begin{align*}
T=\left(
          \begin{array}{cccccccc}
            -\alpha_1-\beta_1 & \alpha_1 &  &&  \\
            \beta_2 &  -\alpha_2-\beta_2& \alpha_2 &  &  \\
             & \beta_3 &  -\alpha_3-\beta_3 &\ddots &  \\
             &  &  \ddots &  \ddots&\alpha_{l-1} \\
               &  &  & \beta_{l} &-\alpha_{l}-\beta_{l}
          \end{array}
        \right).
\end{align*}
Then $P(0)=I_l$ and  (\ref{random-walk}) is equivalent to
\begin{align*}
P(j+1)=P(j)(I_l+T),\;j\geq0.
\end{align*}
So $P(j+1)=(I_l+T)^j$, $j\geq 0$. However, the form $(I_l+T)^j$ provides little information on asymptotic form of  $P(j)$ for large $j$.
Instead, in the spectral theory, the eigenvalues and corresponding eigenfunctions of $T$ play an important role in studying properties of  $P(j)$ for any $j\geq1$. To find an eigenvalue $\lambda$ and the corresponding eigenfunction $(y_1,y_2,\cdots,y_{l})$ of $-T$, we need to study the self-adjoint discrete Sturm-Liouville equation
\begin{align*}
-\nabla(g_j\Delta y_j)=\lambda a_jy_j, \;1\leq j\leq l,
\end{align*}
with the boundary condition $y_{0}=y_{l+1}=0$, where  $g_j=\alpha_ja_j$ and $g_{j-1}=\beta_j a_j$.

Motivated by these two interesting models and recent interest on discrete equations \cite{Bohner1,Clark1,Jirari}, in this paper we consider  a general self-adjoint discrete $d$-dimensional Sturm-Liouville problem  for any $d\geq1$. It
consists of a  symmetric  discrete  Sturm-Liouville equation
\begin{equation}\label{Sturm-Liouville equation}
  -\nabla (P_i\Delta y_i)+Q_iy_i=\ld W_iy_i,\quad 1\leq i\leq N,
\end{equation}
and a self-adjoint boundary condition
\begin{equation}\label{boundary condition }
  A\left(
     \begin{array}{c}
       -y_0 \\
       y_N \\
     \end{array}
   \right)
   +B\left(
     \begin{array}{c}
       P_0\Delta y_0 \\
       P_N\Delta y_N \\
     \end{array}
   \right)=0,
\end{equation}
where $P=\{P_j\}_{j=0}^{N}$,
$Q=\{Q_i\}_{i=1}^{N}$ and $W=\{W_i\}_{i=1}^{N}$ are sequences of  $d\times d$ complex-valued matrices and satisfy
\begin{equation}\label{3}
 P_j,Q_i, W_i\; {\rm are\,\, Hermitian, \,\, } P_j\,{\rm is\,\,invertible, \,\,} W_i\,\, {\rm is \,\,positive\,\,  definite},
\end{equation}
 $0\leq j\leq N$, $\ld\in \C$ is the spectral parameter,  $N\ge 2$, $A$  and  $B$  are $2d\times 2d$ complex-valued  matrices such that
\begin{equation}\label{4}
  \textrm{rank}\,(A, B)=2d,\quad AB^*=BA^*.
\end{equation}

The spectrum of a self-adjoint discrete  Sturm-Liouville problem consists of real and finite eigenvalues, and thus can  be arranged in  the non-decreasing order.
The $n$-th eigenvalue can be considered as a function defined on the space of self-adjoint discrete  Sturm-Liouville problems or on its subset.  This function is not continuous in general, see  the $1$-dimensional case in \cite{Zhu2}. Here we call the set of all  discontinuity points in the considered space  to be the singular set, and call any element in the singular set to be a singular point.
The $n$-th eigenvalue exhibits jump phenomena near the singular points. Unlike only jumping to $-\infty$ in the continuous case, the $n$-th eigenvalue also blows up to $+\infty$ in the discrete case.

The aim of this paper is to determine the singular set and to completely provide the asymptotic behavior of the $n$-th eigenvalue near any fixed singular point for the discrete Sturm-Liouville problems.
As applications, we consider the Sturm-Liouville problem of Atkinson type, transform it into a discrete Sturm-Liouville problem, and then
apply the discrete method to completely characterize singularity of the $n$-th eigenvalue for the Atkinson type.
Though the $n$-th eigenvalue  jumps to $\pm\infty$ near the singular points in the Atkinson type as well as in the discrete case,
  the singular set  in the Atkinson type  is the same one as in the continuous case and is independent of coefficients of the Sturm-Liouville equations.  This leads tremendous difference with the  discrete case, where the singular set is involved heavily  with coefficients of the equations.

Singularity of the $n$-th eigenvalue of Sturm-Liouville problems has attracted a lot of attention (see \cite{Everitt1,Hu-Liu-Wu-Zhu2018,Kato1984,Kong2,Kong-Zettl1996,Zettl1,Zhu2} and their references) since Rellich \cite{Rellich1}.
Let us mention three  contributions to finding  the singular set of the $n$-th eigenvalue and
providing all the asymptotic behavior near each singular point.
 Kong, Wu, and Zettl completely characterized it for  the continuous  $1$-dimensional Sturm-Liouville problems,
while  Hu {\it et al.} gave the answer for  the continuous  $d$-dimensional case, where $d\geq2$.
Zhu and Shi obtained the desired result  for  the discrete $1$-dimensional case.
This paper is devoted to the discrete  case in any dimension.
We mention here that our result in Theorem \ref{main result discrete case}  for singularity of the $n$-th eigenvalue on the boundary
conditions is complete, while the conclusion in Theorem \ref{main result equation discrete case}  for singularity on the equations is partial due to the non-degenerate assumption
(\ref{non-degenerate assumption1})--(\ref{non-degenerate assumption2}).

Compared with the continuous Sturm-Liouville problems, the $n$-th eigenvalue in the discrete case is not continuously dependent on
the equations, and the criterion for  continuity of the $n$-th eigenvalue is  different due to the finiteness of the number of eigenvalues. This makes  the method used in the continuous case \cite{Hu-Liu-Wu-Zhu2018, Kong2}
 unable to apply to the  discrete case.
On the other hand, compared with the $1$-dimensional discrete case, the first difficulty for any dimensional case  is how to divide areas in layers of the considered space such that the $n$-th eigenvalue has the same singularity  in any given area.
Our method in this paper is to find some invertible elementary transformations
 converting the matrix, which determines the number of eigenvalues of the Sturm-Liouville problems, to  a Hermitian matrix.
The areas' division is then determined by the spectral information  of  this Hermitian matrix. The second difficulty is how to prove the asymptotic behavior of the $n$-th eigenvalue.
 Our approach  is
 first to prove the asymptotic behavior in a certain direction using the monotonicity of  continuous eigenvalue branches, and then
combine  the local topological property (geometric structure)  of the considered  space  with the perturbation theory of eigenvalues  to obtain  the whole asymptotic behavior.  This can be regarded as a generalization of
the method developed for $1$-dimensional discrete  case in
\cite{Zhu2} to any dimension.
Finally, though our method for the Atkinson type is by transforming the Sturm-Liouville problem   into a discrete one,  it turns out to be no singularity of the $n$-th eigenvalue on the equations for the Atkinson type.

The rest of this paper is organized as follows.
In Section 2, topology on the space of Sturm-Liouville equations, and that on the space of boundary
conditions are presented.
Properties of eigenvalues are given in Section 3. The number and multiplicity of eigenvalues are discussed in Subsection 3.1, continuous eigenvalue branches are constructed  and their properties are provided in Subsection 3.2, and  properties of the $n$-th eigenvalue are presented in Subsection 3.3.
In Section 4,
singularity of the $n$-th eigenvalue on the boundary conditions is completely characterized for a fixed equation in Subsection 4.1, while
singularity of the $n$-th eigenvalue on the equations is obtained for a fixed boundary condition under a non-degenerate assumption in Subsection 4.2. Sturm-Liouville problem of the Atkinson type is transformed to a discrete one, and singularity of the $n$-th eigenvalue is provided thoroughly in Section 5. Conclusions are given in Section 6.

\bigskip

\noindent\textbf{Notation.}

\bigskip

By $\R$ and $\C$ denote the set of all the real  and complex numbers, respectively.
The set of all $m\times n$ matrices over a field $\mathbb{F}$ is denoted by $\mathcal{M}_{m,n}(\mathbb{F})$,
and $\mathcal{M}_{n,n}(\mathbb{F})$ is abbreviated to $\mathcal{M}_n(\mathbb{F})$.
  $A^*$ is the complex conjugate transpose of $A\in \mathcal{M}_{m,n}(\mathbb{F})$, while $A^T$ is the  transpose of $A$.
$\mathcal{H}_{n}(\mathbb{F})$ is the set of all  $n\times n$ Hermitian matrices, while
$\mathcal{P}_{n}(\mathbb{F})$ is the set of all  $n\times n$ positive definite matrices over a field $\mathbb{F}$.
For a matrix  $S\in \mathcal{M}_{n}(\mathbb{F})$,  its entries and columns are
denoted by $s_{ij}$ and $s_j = (s_{1j},\cdots, s_{nj})^T$, respectively, $1\le i,j\le n$.
By $I_n$ denote the $n\times n$ unit matrix.  $\sharp(K)$ is the cardinality of the set $K$.
By $r^-(A)$, $r^0(A)$, and $r^+(A)$  denote the total multiplicity of  negative,
zero, and positive eigenvalues of  $A\in\mathcal{H}_n(\mathbb{C})$, respectively.

\section{Space of self-adjoint discrete   Sturm-Liouville problems}

In this section, we introduce the topology on the  space of self-adjoint discrete   Sturm-Liouville problems.

The space of discrete Sturm-Liouville equations is
\begin{equation*}
 \Omega_N^{\C} := \{\pmb\omega=\left(\{P_j^{-1}\}_{j=0}^{N}, \{Q_i\}_{i=1}^{N}, \{W_i\}_{i=1}^{N}\right)\in \left(\mathcal{M}_{d}(\C)\right)^{3N+1}: \, (\ref{3}) \;{\rm holds}\}
\end{equation*}
with the topology induced by $\C^{(3N+1)d^2}$.


Note that the space of self-adjoint boundary conditions is the same as the continuous case. Following \cite{Hu-Liu-Wu-Zhu2018},
it is exactly the quotient space
\begin{equation}\label{the space of self-adjoint boundary conditions}
 \mathcal{ B}^{\C} :=\lower3pt\hbox{${\rm
GL}(2d,\mathbb C)$}\backslash \raise2pt\hbox{$\mathcal{L}_{2d,4d}(\mathbb C)$},
\end{equation}
where
\begin{eqnarray*}
  \mathcal{L}_{2d,4d}(\C) :=\{(A, B)\in \mathcal{M}_{2d,4d}(\C): \, {\rm rank}(A, B)=2d, AB^*=BA^*\}
\end{eqnarray*}
and
\begin{equation*}
  {\rm GL}(2d,\C) := \{T\in \mathcal{M}_{2d}(\C) : \det T \ne  0\}.
\end{equation*}
The boundary condition in  $\mathcal{ B}^{\C}$ is denoted by  $[A\,|\,B]:=\{(TA\;|\;TB):T\in {\rm GL}(2d,\mathbb{C})\}$. Bold faced capital Latin letters, such as $\mathbf{A}$, are also used  for  boundary conditions.

Next we introduce the following form for the local coordinate systems on $\mathcal{B}^{\C}$.
Let $K$ be any subset of $\{1,2,\cdots,2d\}$.
 Denote
\begin{equation}\label{K1K2-definition}
  K_1=K\cap\{1,2,\cdots,d\},\quad  K_2=K\cap\{d+1,d+2,\cdots,2d\}.
\end{equation}
By $E_K$ denote the $4d\times4d$ matrix generated from $I_{4d}$ by multiplying $-1$ to the  $(k+2d)$-th column
and then exchanging the $k$-th and the $(k+2d)$-th columns  for each $k\in K$.
Then it has the following form:
\begin{equation}\label{construction of EK}
  E_{K}=\begin{pmatrix}E_{K,1}\\E_{K,2}\end{pmatrix}=\left(
          \begin{array}{cccccccc}
            E_{1} & 0 & I_d-E_{1} & 0 \\
            0 &  E_{2}& 0 &  I_d-E_{2}  \\
            E_{1}-I_d & 0 &  E_{1} & 0  \\
            0 &  E_{2}-I_d & 0 &  E_{2} \\
          \end{array}
        \right),
\end{equation}
where $E_{K,1},E_{K,2}\in\mathcal{M}_{2d,4d}(\mathbb{C})$,  $E_1=\{\af_1,\af_2,\cdots,\af_d\}$ and $E_2=\{\bt_1,\bt_2,\cdots,\bt_d\}$ are $d\times d$ diagonal  matrices with
\begin{align}\label{E1E2}
  \af_i=\left\{\begin{array}{cc}
          0 & {\rm if} \;i\in K_1, \\
          e_i & {\rm if} \;i\notin K_1,
        \end{array}\right.\quad
        \bt_i=\left\{\begin{array}{cc}
          0  & {\rm if} \;d+i\in K_2, \\
         e_i & {\rm if} \;d+i\notin K_2,
        \end{array}\right.
\end{align}
and $e_i$  is  the $i$-th column of    $I_d$.
Then
\begin{align}\label{EKrelation}
E_K^*J_{2d}E_K=J_{2d},\;\;E_KE_K^*=I_{4d},
\end{align}
where
\begin{align*}
{J}_{2d}=\left(
      \begin{array}{cc}
        0 & -I_{2d} \\
        I_{2d} & 0 \\
      \end{array}
    \right).
\end{align*}
We define
\begin{align}\label{co-sys}
\mathcal{ O}_{K}^{\C}:=\{[(S\,|\, I_{2d})E_{K}]:\; S\in \mathcal{H}_{2d}(\C)\}.
\end{align}
For $\mathbf{A}=[(S\,|\, I_{2d})E_{K}]\in \mathcal{ O}_{K}^{\C}$, we  denote $S$ by $S(\mathbf{A})$ to indicate its dependence on $\mathbf{A}$ if necessary. It is clear that $\mathcal{ O}_{K}^{\C}$ defined here coincides with that defined  in (2.1) of \cite{Hu-Liu-Wu-Zhu2018}.
It follows from Theorem 2.1 in \cite{Hu-Liu-Wu-Zhu2018} that
\begin{equation*}
  \mathcal{B}^{\C}=\bigcup_{K\subset \{1,2,\cdots,2d\}}  \mathcal{ O}_{K}^{\C}.
\end{equation*}
Moreover, $\mathcal{B}^{\mathbb{C}}$ is a connected and compact real-analytic manifold of dimension $4d^2$. The readers are also referred to \cite{Arnold1967,Arnold2000,Kong3,Long2002} for more details.

The product space $ \Omega_N^{\C} \times\mathcal{ B}^{\C} $ is the space of self-adjoint discrete Sturm-Liouville problems, and
$(\pmb\omega, \mathbf{A})$  is used to stand for an element in
$\Omega_N^{\C}  \times\mathcal{ B}^{\C} $ in the sequel.

\section{Properties of eigenvalues}

In this section, we study  properties of  eigenvalues of the self-adjoint discrete Sturm-Liouville problems.

\subsection{The number and multiplicity  of  eigenvalues.}
Let
\begin{equation*}
  l[0,N+1]:=\left\{y=\{y_i\}_{i=0}^{N+1}:\; y_i\in \C^{d}, \,0\leq i\leq N+1\right\}.
\end{equation*}
The initial value problem of (\ref{Sturm-Liouville equation}) has a unique solutions. More precisely,

\begin{lemma} Let  $z_{i_0}, \tilde z_{i_0}\in \C^{d}$ for some $1\leq i_0\leq N$. Then, for each $\ld\in\C$,  (\ref{Sturm-Liouville equation}) has a unique solution $y(\ld) \in l[0, N + 1]$ satisfying
$y_{i_0}(\ld) = z_{i_0}$, $P_{i_0}\Delta y_{i_0}(\ld) = \tilde z_{i_0}$.
\end{lemma}

\begin{proof}  This can be deduced by the invertibility of $P_{j}$ for $0\leq j\leq N$ and the iteration of
\begin{equation*}
P_i\Delta y_i= P_{i-1}\Delta y_{i-1}-(\ld W_i -Q_i)y_i,\quad 1\leq i\leq N.
\end{equation*}
\end{proof}

Recall that  $\ld$ is called an eigenvalue of the discrete Sturm-Liouville problem $(\pmb\omega, \mathbf{A})$
 if there exists $y \in l[0, N + 1] $ which is non-trivial and solves (\ref{Sturm-Liouville equation})--(\ref{boundary condition }). Here  $y $ is called an eigenfunction corresponding to $\ld$, and it is said to be normalized  if
$ \sum_{i=1}^Ny^*_iW_iy_i=1.$
By $\sigma(\pmb\omega, \mathbf{A})$ denote the spectral set of $(\pmb\omega, \mathbf{A})$.
For  any $\ld\in \C$, let $\phi^{j}(\ld)=\{\phi_i^{j}(\ld)\}_{i=0}^{N+1}$, $j=1,\cdots,2d$, be the fundamental solutions to (\ref{Sturm-Liouville equation})
determined by the initial data
\begin{equation*}
  \left(
    \begin{array}{ccc}
      \phi^{1}_0(\ld) & \cdots & \phi^{2d}_0(\ld) \\
      P_0\Delta\phi^{1}_0(\ld) & \cdots & P_0\Delta\phi^{2d}_0(\ld) \\
    \end{array}
  \right)=I_{2d}.
\end{equation*}
Denote
\begin{equation*}
  \Phi(\ld):=\left(
    \begin{array}{ccc}
      -\phi^{1}_0(\ld) & \cdots & -\phi^{2d}_0(\ld) \\
      \phi^{1}_N(\ld) & \cdots & \phi^{2d}_N(\ld) \\
    \end{array}
  \right),
  \quad
\Psi(\ld):=  \left(
\begin{array}{ccc}
      P_0\Delta\phi^{1}_0(\ld) & \cdots & P_0\Delta\phi^{2d}_0(\ld) \\
      P_N\Delta\phi^{1}_N(\ld) & \cdots & P_N\Delta\phi^{2d}_N(\ld) \\
    \end{array}
  \right).
\end{equation*}
Then the eigenvalues of $(\pmb\omega, \mathbf{A})$  can be regarded as zeros of the polynomial $\Gamma_{(\pmb\omega, \mathbf{A})}$ as follows.

\begin{lemma} \label{zero} $\ld\in \sigma(\pmb\omega, \mathbf{A})$ if and only of $\ld$ is a zero
of
\begin{equation*}
  \Gamma_{(\pmb\omega, \mathbf{A})}(\ld):=\det(A\Phi(\ld)+B \Psi(\ld)).
\end{equation*}
\end{lemma}

\begin{proof} The proof is similar as that of Lemma 3.2 in \cite{Zhu1}.
\end{proof}

Let $\lambda\in\sigma(\pmb\omega,\mathbf{A})$.
The order of $\lambda$ as a zero of $\Gamma_{(\pmb\omega,\mathbf{A})}$ is called its analytic multiplicity. The number of linearly independent eigenfunctions for $\lambda$ is called its geometric multiplicity.
Let $x_i=P_i\Delta y_i$ for $0\leq i\leq N$. Then the Sturm-Liouville  equation (\ref{Sturm-Liouville  equation}) can be transformed to
a discrete linear Hamiltonian system:
\begin{align*}
J_d\Delta \begin{pmatrix}y_{i}\\x_i\end{pmatrix}=\left(\begin{pmatrix}-Q_{i+1}&0\\0&P_i^{-1}\end{pmatrix}+\lambda \begin{pmatrix}W_{i+1}&0\\0&0\end{pmatrix}\right)R\begin{pmatrix}y_{i}\\x_i\end{pmatrix}, \;\;0\leq i\leq N-1,
\end{align*}
where  $R(y^T_{i},x^T_i)^T=(y^T_{i+1},x^T_i)^T$ is the partial right shift operator.
Then by Theorem 4.1  in \cite{Zhu5}, we get the relationship of analytic and geometric multiplicities of $\lambda$:

\begin{lemma}\label{multiplicity}
 The analytic and geometric multiplicities of  $\lambda\in\sigma(\pmb\omega,\mathbf{A})$ are the same.
\end{lemma}

Therefore, we do not distinguish these two  multiplicities of $\lambda$. Let   $\sharp_1(\sigma(\pmb\omega,\mathbf{A})\cap I)$  be the number of eigenvalues  in  $I\subset \mathbb{R}$, counting multiplicities, of $(\pmb\omega,\mathbf{A})$. Since $\sigma(\pmb\omega,\mathbf{A})\subset\mathbb{R}$ by \cite{Shi3}, we have   $\sharp_1(\sigma(\pmb\omega,\mathbf{A})\cap \mathbb{R})=\sharp_1(\sigma(\pmb\omega,\mathbf{A}))$.
The next lemma determines $\sharp_1(\sigma(\pmb\omega,\mathbf{A}))$.

\begin{lemma}\label{number of eigenvalues}
\begin{align}\label{number}
\sharp_1(\sigma(\pmb\omega,\mathbf{A}))=(N-2)d+{\rm rank} (A_1P_0^{-1}+B_1,B_2),
\end{align}
where $A_1,B_j\in\mathcal{M}_{2d\times d}$ $(j=1,2)$ are given by
\begin{align}\label{A1A2B1B2}
\mathbf{A}=[A\;|\;B]=[(A_1,A_2)|(B_1,B_2)].\end{align}
\end{lemma}
\begin{proof}
By Theorem 4.1 in \cite{Shi3},
$$\sharp_1(\sigma(\pmb\omega,\mathbf{A}))=(N-2)d+{\rm rank} (A_1+B_1P_0,B_2).$$
Then (\ref{number}) is obtained by $(A_1+B_1P_0,B_2)\begin{pmatrix}P_0^{-1}&0\\0&I_d\end{pmatrix}=(A_1P_0^{-1}+B_1,B_2)$.
\end{proof}
Note that $(N-2)d\leq \sharp_1(\sigma(\pmb\omega,\mathbf{A}))\leq Nd$.

\subsection{Continuous Eigenvalue Branch}
In this subsection, we  construct  continuous eigenvalue branches. Then we  study their derivative formulae and  monotonicity in some directions.

The first lemma is the small perturbation theory of eigenvalues.

\begin{lemma}\label{small perturbation} Let $(\pmb\omega_0, \mathbf{A}_0)\in\mathcal{O}\subset  \Omega_N^{\C} \times\mathcal{ B}^{\C}$, and $c_1, c_2\in\mathbb{R}\setminus\sigma(\pmb\omega_0, \mathbf{A}_0)$ with $c_1 < c_2.$
Then there exists a neighborhood $\mathcal{U}\subset\mathcal{O}$ of $(\pmb\omega_0, \mathbf{A}_0)$  such that for each  $(\pmb\omega, \mathbf{A})\in\mathcal{U}$,  $\sharp_1(\sigma(\pmb\omega, \mathbf{A})\cap(c_1, c_2))=\sharp_1(\sigma(\pmb\omega_0, \mathbf{A}_0)\cap(c_1, c_2))$ and  $c_1, c_2\notin\sigma(\pmb\omega, \mathbf{A})$.
\end{lemma}

\begin{proof}
Using Lemma \ref{zero}, the proof is by a standard perturbation procedure  for zeros of the analytic function $\Gamma_{(\pmb\omega_0, \mathbf{A}_0)}$.
\end{proof}
By Lemma \ref{small perturbation} and a similar approach to Theorem 3.5 in \cite{Zhu1}, we then construct the continuous eigenvalue branches.

\begin{lemma}\label{continuous eigenvalue branch}
 Let  $(\pmb\omega_0, \mathbf{A}_0)\in \Omega_N^{\C} \times\mathcal{ B}^{\C}$ and $\ld_*\in \sigma(\pmb\omega_0, \mathbf{A}_0)$ with multiplicity $m$. Fix a small $\varepsilon>0$ such that $\sigma(\pmb\omega_0, \mathbf{A}_0)\cap[\ld_*-\varepsilon,\ld_*+\varepsilon]=\{\ld_*\}$.
 Then there is a connected neighborhood $\mathcal{U}$ of $(\pmb\omega_0, \mathbf{A}_0)$ and continuous functions  $\Lambda_i:\mathcal{U}\to \R$, $1\le i\le m$, such that   $\ld_*-\varepsilon<\Ld_1(\pmb\omega,\mathbf{A})\leq \cdots\leq \Ld_m(\pmb\omega,\mathbf{A}) <\ld_*+\varepsilon$ and $\ld_*\pm\varepsilon\notin\sigma(\pmb\omega, \mathbf{A})$ for all
$(\pmb\omega, \mathbf{A})\in \mathcal{U}$, where $\{\Lambda_i(\pmb\omega_, \mathbf{A})\}_{i=1}^m\subset\sigma(\pmb\omega, \mathbf{A})$.
\end{lemma}

Here $\Lambda_i:\mathcal{U}\to \R$, $1\le i\le m$, are called the continuous eigenvalue branches through $\ld_*$.
We write $\Lambda_i(\pmb\omega)$ when $\mathbf{A}$ is fixed, and write $\Lambda_i(\mathbf{A})$ when $\pmb\omega$ is fixed.
Then we shall make a continuous choice of eigenfunctions for the eigenvalues along a  continuous simple eigenvalue branch $(m=1)$.

\begin{lemma}\label{continuous choice of eigenfunctions}
  Let $u_0$ be an eigenfunction for a simple eigenvalue $\lambda_*\in\sigma(\pmb\omega_0,\mathbf{A}_0)$,
 and
$\Lambda$ be the continuous eigenvalue branch defined on $\mathcal{U}$ through $\lambda_*$.  Then
 there exists a neighborhood $\mathcal{U}_1\subset\mathcal{U}$ of $(\pmb\omega_0,\mathbf{A}_0)$ such that   for any $(\pmb\omega,\mathbf{A})\in\mathcal{U}_1$,
  there is an eigenfunction $u_{\Lambda(\pmb\omega,\mathbf{A})}$ for $\Lambda(\pmb\omega,\mathbf{A})$  satisfying that
 $u_{\Lambda(\pmb\omega,\mathbf{A})}=u_0$, and
$u_{\Lambda(\pmb\omega,\mathbf{A})}\to u_{\Lambda(\pmb\omega_0,\mathbf{A}_0)}$  in $\mathbb{C}^{(N+2)d}$
as $\mathcal{{U}}_1\ni(\pmb\omega,\mathbf{A})\to(\pmb\omega_0,\mathbf{A}_0)$.
\end{lemma}

\begin{proof}
The proof is similar to that of Lemma 4.3 in \cite{Zhu1}, and thus we omit the details.
\end{proof}

Besides Lemma \ref{continuous choice of eigenfunctions}, we also need the following lemma to deduce the derivative formulae for continuous simple eigenvalue branches.

\begin{lemma}\label{sbc} Let $y$ be an eigenfunction for
 $\lambda\in\sigma(\pmb\omega,\mathbf{A})$ and $z$ be an eigenfunction for  $\tilde\lambda\in\sigma({\pmb{\tilde\omega}},\mathbf{A})$, where $\pmb\omega=(P^{-1},Q,W),{\pmb{\tilde\omega}} =(\tilde{P}^{-1},\tilde{Q},\tilde{W})$ and $\mathbf{A}=[A\,|\,B]$. Then
\begin{equation}\label{same boundary}
  (\Delta z_0)^*\tilde{P}_0y_0-z_0^*P_0\Delta y_0=(\Delta z_N)^*\tilde P_N y_N-z_N^*{P}_N\Delta y_N.
\end{equation}
\end{lemma}

\begin{proof} For convenience, denote
\begin{equation}\label{YZ-def}
 \quad (A,\,B)=\left(
                              \begin{array}{c}
                                \af_1 \\
                                \af_2 \\
                                \vdots \\
                                \af_{2d} \\
                              \end{array}
                            \right),\quad Y=\left(
                                              \begin{array}{c}
                                                -y_0 \\
                                                y_N\\
                                                 P_0\Delta y_0\\
                                                P_N\Delta y_N \\
                                              \end{array}
                                            \right),
                                            \quad Z=\left(
                                              \begin{array}{c}
                                                -z_0 \\
                                               z_N \\
                                                 \tilde P_0\Delta z_0 \\
                                               \tilde P_N\Delta z_N \\
                                              \end{array}
                                            \right),
\end{equation}
where  $\alpha_i\in\mathcal{M}_{1,4d}(\mathbb{C})$, $i=1,\cdots,2d$. Then
\begin{equation}\label{bc}
  (A,\,B){J_{2d}}(A,\,B)^*=0, \quad (A,\,B)Y=0,\quad (A,\,B)Z=0.
\end{equation}
Since ${\rm rank}(A,\,B)=2d$, the first equation in (\ref{bc}) yields that  each solution of the equation $(A,\,B)X=0$ is a linear combination
of ${J}_{2d}\af_i^*$, $1\leq i\leq 2d$.
From the last two equations  in (\ref{bc}),
we know that there exists $c_i, d_i\in\mathbb{C}$, $1\leq i \leq 2d$, such that  $Y=\sum_{i=1}^{2d}c_iJ_{2d}\af_i^*$ and $Z=\sum_{i=1}^{2d}d_iJ_{2d}\af_i^*$.
The first equation in (\ref{bc}) also implies that
\begin{equation*}
  \af_iJ_{2d}\af_j^*=0,\quad 1\le i,j\le 2d.
\end{equation*}
So
\begin{equation*}
  Z^*{J_{2d}}Y=(\sum_{i=1}^{2d}d_i{J_{2d}}\af_i^*)^*{J_{2d}}(\sum_{i=1}^{2d}c_i{J_{2d}}\af_i^*)=0,
\end{equation*}
which is equivalent to (\ref{same boundary}). The proof is complete.
\end{proof}

Note that the method used in Lemma 4.4 of \cite{Zhu2} depends on separated and coupled boundary conditions, and thus
can not be applied to Lemma \ref{sbc} here  for  mixing boundary conditions when  $d\ge 2$.
With the help of Lemma \ref{sbc}, we give the derivative formulae of the continuous simple eigenvalue branch with respect to coefficients of the Sturm-Liouville equations.

\begin{lemma}\label{equation-iden} Fix $\mathbf{A}\in \mathcal{B}^{\C}$.
Let ${\pmb\omega}=(P^{-1},Q,W)\in \Omega_N^{\C}$,  $\ld_*$ be a simple eigenvalue of $(\pmb\omega,\mathbf{A})$, $y\in  l[0, N+1]$ be a normalized eigenfunction for $\ld_*$,
and $\Ld$ be the continuous simple eigenvalue branch over $\Omega_N^{\C}$ through $\ld_*$. Then
\begin{align}\label{derivative-equation}
d \Ld |_{\pmb\omega}(H,K,L)=-\sum_{i=0}^{N-1}(P_i\Delta y_i)^*H_i(P_i\Delta y_i)+\sum_{i=1}^Ny_i^* K_iy_i
-\ld_*\sum_{i=1}^Ny_i^* L_iy_i
\end{align}
for all $(H,K,L)=((H_0,\cdots,H_N), (K_1,\cdots,K_N), (L_1,\cdots,L_N))\in (\mathcal{H}_{d}(\mathbb{C}))^{3N+1}$ and $(P^{-1}+H,Q+K,W+L)\in \Omega_N^{\C}$.
\end{lemma}

\begin{proof}
Let $\pmb{\sigma}\in \Omega_N^{\C}$ with $\pmb\sigma=({P}^{-1}+H,{Q}+K,{W}+L)=:(\tilde P^{-1},\tilde Q,\tilde W)$.
By Lemma \ref{continuous choice of eigenfunctions}, we can choose an eigenfunction $z = z(\cdot, \pmb{\sigma})$ for $\Ld=\Ld(\pmb{\sigma})$
with $\pmb{\sigma}$  sufficiently close to $\pmb\omega$ in $\Omega_N^{\C}$  such that $z\to y$ as $\pmb{\sigma}\to \pmb\omega$.
Then it follows from (\ref{Sturm-Liouville equation}) that
\begin{align*}
&[\Ld(\pmb\sigma) - \Ld(\pmb\omega)]\sum_{i=1}^Nz_i^* W_iy_i\\
=&\sum_{i=1}^N[z_i^*\nabla(P_i\Delta y_i)-(\nabla(\tilde{P}_i\Delta z_i))^*y_i]-\Ld(\pmb\sigma)\sum_{i=1}^Nz_i^* L_iy_i+\sum_{i=1}^Nz_i^* K_iy_i\\
=&-\sum_{i=0}^{N-1}\Delta z_i^*P_i\Delta y_i+\sum_{i=1}^N\nabla(z_i^*P_i\Delta y_i)+
\sum_{i=0}^{N-1}\Delta z_i^*\tilde{P}_i\Delta y_i -\sum_{i=1}^N\nabla(\Delta z_i^*\tilde{P}_iy_i)\\
&-\Ld(\pmb\sigma)\sum_{i=1}^Nz_i^* L_iy_i+\sum_{i=1}^Nz_i^* K_iy_i\\
=&\sum_{i=0}^{N-1}(\tilde{P}_i\Delta z_i)^*(P_i^{-1}-\tilde{P}_i^{-1})(P_i\Delta y_i)+z_N^*P_N\Delta y_N-z_0^*P_0\Delta y_0-(\Delta z_N)^*\tilde{P}_Ny_N+(\Delta z_0)^*\tilde{P}_0 y_0\\
&-\Ld(\pmb\sigma)\sum_{i=1}^Nz_i^* L_iy_i+\sum_{i=1}^Nz_i^* K_iy_i.
 \end{align*}
By Lemma \ref{sbc}, we  get
\begin{align*}
&[\Ld(\pmb\sigma) - \Ld(\pmb\omega)]\sum_{i=1}^Nz_i^* W_iy_i\\
=&-\sum_{i=0}^{N-1}(\tilde{P}_i\Delta z_i)^*H_i(P_i\Delta y_i)
-\Ld(\pmb\sigma)\sum_{i=1}^Nz_i^* L_iy_i+\sum_{i=1}^Nz_i^* K_iy_i,
 \end{align*}
which yields that (\ref{derivative-equation}) holds. This completes the proof.
\end{proof}

Let us fix all the components of $\pmb\omega$ except  $P_j^{-1}$,  and write the perturbed term $P_j^{-1}$ by $P_j^{-1}(\pmb\omega)$ to indicate its dependence on $\pmb\omega$ for a given $0\leq j\leq N-1$. $Q_i(\pmb\omega)$ has the similar meaning for $1\leq i\leq N$.
The we get the following monotonicity  result.

\begin{corol}\label{monotonicity equation}  Fix $\mathbf{A}\in \mathcal{B}^\mathbb{C}$. Let $\Lambda$ be a continuous eigenvalue branch defined on $\mathcal{V}\subset \Omega_N^{\C}$.
If  $  P_j^{-1}(\pmb\sigma)-P_j^{-1}(\pmb\omega)$ is positive semi-definite for a given $0\leq j\leq N-1$, then $\Lambda(\pmb\sigma)\leq\Lambda(\pmb\omega)$.
If  $  Q_i(\pmb\sigma)-Q_i(\pmb\omega)$ is positive semi-definite for a given $1\leq i\leq N$, then $\Lambda(\pmb\omega)\leq\Lambda(\pmb\sigma)$.
\end{corol}

Then we give the  derivative formula of a continuous simple eigenvalue
branch with respect to boundary conditions.

\begin{lemma}\label{derivative formula1}
Fix $\pmb\omega\in \Omega_N^\mathbb{C}$. Let $\ld_*$ be a simple eigenvalue of $(\pmb\omega, \mathbf{A})$
for some $\mathbf{A}\in \mathcal{O}_{K}^\mathbb{C}$,
$y \in l[0, N + 1]$ be a normalized eigenfunction for $\ld_*$, and $\Ld$
be the continuous simple eigenvalue branch  through $\ld_*$.
Then
\begin{equation*}
  d \Ld |_{\mathbf{A}}(H)=Y^*E_{K,1}^*H E_{K,1}Y
\end{equation*}
for $H\in\mathcal{H}_{2d}(\mathbb{C})$, where
$E_{K,1}$ and $Y$ are given in (\ref{construction of EK}) and (\ref{YZ-def}), respectively.
\end{lemma}

\begin{proof} By (\ref{co-sys}), there exists $S\in\mathcal{H}_{2d}(\mathbb{C})$ such that $\mathbf{A}=[S\;|\;I_{2d}]E_K$. Let $\mathbf{B}=[S+H\;|\;I_{2d}]E_K$ with $H\in\mathcal{H}_{2d}$.
   Then there exists an eigenfunction $\tilde y=y_{\Lambda(\mathbf{B})}$ for $\Lambda(\mathbf{B})$ such that $\tilde y\to y$  in $\mathbb{C}^{(N+2)d}$ as $\mathbf{B}\to \mathbf{A}$.
$\tilde Y$ has the similar meaning as $Y$.
Note that  $\tilde y$ and ${y}$ satisfy
\begin{align*}
-\nabla(P_i\Delta\tilde{y}_i)+Q_i\tilde{y}_i=\Lambda(\mathbf{B})W_i\tilde{y}_i,\;\;
-\nabla(P_i \Delta y_i)+Q_i y_i=\Lambda(\mathbf{A})W_i y_i,\;1\leq i\leq N,
\end{align*}
and thus
\begin{align*}
(\Lambda(\mathbf{B})-\Lambda(\mathbf{A}))\tilde y_i^* W_i y_i
=-\Delta[y_{i-1},\tilde y_{i-1}],
\end{align*}
where $[y_i,\tilde y_i]=(\Delta \tilde y_i)^*P_iy_i-\tilde y_i^*P_i\Delta y_i$.
From the boundary conditions  $\mathbf A$ and $\mathbf B$,
we have
\begin{align}\label{boundary condition deriveSE}
SE_{K,1}Y+E_{K,2}Y=0\;{\rm and}\;\;
(S+H)E_{K,1}\tilde Y+E_{K,2}\tilde Y=0.
\end{align}
It then follows from (\ref{EKrelation}) and (\ref{boundary condition deriveSE}) that
\begin{align*}
(\Lambda(\mathbf{B})-\Lambda(\mathbf{A}))\sum_{i=1}^N\tilde y_i^* W_i y_i
=&[y_{0},\tilde y_{0}]-[y_{N},\tilde y_{N}]=\tilde Y^*J_{2d}^*{Y}\\
=&\tilde Y^*E_K^*J_{2d}^*E_K{Y}
=-\tilde Y^*E_{K,2}^*E_{K,1}Y+\tilde Y^*E_{K,1}^*E_{K,2}Y\\
=&\tilde Y^*E_{K,1}^*(S+H)E_{K,1}Y-\tilde Y^*E_{K,1}^*SE_{K,1}Y=\tilde Y^*E_{K,1}^*HE_{K,1}Y.
\end{align*}
This completes the proof.
\end{proof}
The following result is a direct consequence of Lemma \ref{derivative formula1}.
\begin{corol}\label{cor-monotonicity-boundary}
 Fix $\pmb\omega\in \Omega_N^\mathbb{C}$. Let $\Lambda$ be a continuous eigenvalue branch defined on $\mathcal{U}\subset \mathcal{O}_K^{\C}$.
Then $\Lambda(\mathbf{A})\leq\Lambda(\mathbf{B})$ if $\mathbf{A},\mathbf{B}\in\mathcal{U}$ and $S(\mathbf{B})-S(\mathbf{A})$ is positive semi-definite.
\end{corol}

\subsection{Properties of the $n$-th eigenvalue.}
Based on Lemma \ref{number of eigenvalues},
the eigenvalues of $(\pmb\omega,\mathbf{A})\in\Omega_N^{\C} \times\mathcal{ B}^{\C}$ can be arranged in the following non-decreasing order:
\begin{align*}
\lambda_1(\pmb\omega,\mathbf{A})\leq\lambda_2(\pmb\omega,\mathbf{A})\leq \cdots\leq \lambda_{\sharp_1(\sigma(\pmb\omega,\mathbf{A}))}(\pmb\omega,\mathbf{A}).
\end{align*}
Therefore, for any $1\leq n\leq Nd$, the $n$-th eigenvalue can be regarded as a function defined on $\Omega_N^{\C} \times\mathcal{ B}^{\C}$ or on its subset, called the $n$-th eigenvalue function.
Firstly, we provide a criterion for all these functions to be
continuous on a subset of $\Omega_N^{\C} \times\mathcal{ B}^{\C}$.

\begin{lemma}\label{continuity criteria} Let $\mathcal{O}$ be a connected subset of $\Omega^{\C}_N\times \mathcal{B}^{\C}$.
If $ \sharp_1(\sigma(\pmb\omega,\mathbf{A}))\equiv k_0$, $(\pmb\omega,\mathbf{A})\in \mathcal{O}$, for some $k_0>0$,
then the restrictions of $\lambda_n$, $1\leq n\leq k_0$, to $\mathcal{O}$ are continuous. Moreover,
they are locally continuous eigenvalue branches on $\mathcal{O}$.
\end{lemma}

Then we list several other properties of the $n$-th eigenvalue function
 in order to study its asymptotic behavior. The following lemma strengths the result in Theorem 2.2 of  \cite{Zhu1}.

\begin{lemma}\label{infinity-finity}
 Let $\mathcal{O}\subset\Omega^{\C}_N\times \mathcal{B}^{\C}$,
 $\sharp_1(\sigma(\pmb\omega,\mathbf{A}))=m_1+m_2+m_3$ for all $(\pmb\omega,\mathbf{A})\in \mathcal{O}$, and $\sharp_1(\sigma(\pmb\omega_0,\mathbf{A}_0))=m_2$  for some $(\pmb\omega_0,\mathbf{A}_0)\in \bar{\mathcal{O}}\setminus \mathcal{O}$, where $ m_i\geq 0$, $1\leq i\leq 3$. If
\begin{equation}\label{asymptotic-infty}
 \lim_{\mathcal{O}\ni(\pmb\omega,\mathbf{A})\to (\pmb\omega_0,\mathbf{A}_0)}\ld_{n}(\pmb\omega,\mathbf{A})=-\infty,\;1\le n\le m_1,
\end{equation}
and
\begin{equation}\label{asymptotic+infty}
 \lim_{\mathcal{O}\ni(\pmb\omega,\mathbf{A})\to (\pmb\omega_0,\mathbf{A}_0)}\ld_{n}(\pmb\omega,\mathbf{A})=+\infty,\;m_1+m_2+1\le n\le m_1+m_2+m_3,
\end{equation}
then
\begin{equation*}
 \lim_{\mathcal{O}\ni(\pmb\omega,\mathbf{A})\to (\pmb\omega_0,\mathbf{A}_0)}\ld_{n}(\pmb\omega,\mathbf{A})=\ld_{n-m_1}(\pmb\omega_0,\mathbf{A}_0),\;m_1+1\le n\le m_1+m_2.
\end{equation*}
\end{lemma}
\begin{proof} Let $c_1,c_2\in\mathbb{R}$ such that $\sharp_1(\sigma(\pmb\omega_0,\mathbf{A}_0)\cap(c_1,c_2))=m_2$. Then we get  by Lemma \ref{small perturbation} that  there exists a neighborhood $\mathcal{U}\subset\mathcal{O}$ of $(\pmb\omega_0,\mathbf{A}_0)$  such that $\sharp_1(\sigma(\pmb\omega,\mathbf{A})\cap(c_1,c_2))=m_2$ and $c_1,c_2\notin \sigma(\pmb\omega,\mathbf{A})$ for all $(\pmb\omega,\mathbf{A})\in\mathcal{U}$. It follows from (\ref{asymptotic-infty})--(\ref{asymptotic+infty}) that $\mathcal{U}$ can be shrunk such that $\sharp_1(\sigma(\pmb\omega,\mathbf{A})\cap(-\infty,c_1))=m_1$ and $\sharp_1(\sigma(\pmb\omega,\mathbf{A})\cap(c_2,\infty ))=m_3$  for all $(\pmb\omega,\mathbf{A})\in\mathcal{U}$. This implies that $\sigma(\pmb\omega,\mathbf{A})\cap(c_1,c_2)=\{\ld_{n}(\pmb\omega,\mathbf{A}):\;m_1+1\le n\le m_1+m_2\}$ for all $(\pmb\omega,\mathbf{A})\in\mathcal{U}$. Then the conclusion holds again by Lemma \ref{small perturbation}.
\end{proof}

\begin{lemma}\label{integer}
 Let $\mathcal{O}$ be a connected subset of $\Omega^{\C}_N\times \mathcal{B}^{\C}$
and $c_1,c_2\in\mathbb{R}$ with $c_1<c_2$.
Assume that for all $(\pmb\omega,\mathbf{A})\in\mathcal{O}$,  $\sharp_1(\sigma(\pmb\omega,\mathbf{A}))=k$,
$\sharp_1(\sigma(\pmb\omega,\mathbf{A})\cap(c_1, c_2))=m$ with $m<k$, and
  $c_1, c_2\notin\sigma(\pmb\omega,\mathbf{A})$.
Then the other $k-m$ eigenvalues out of $[c_1, c_2]$, denoted by
$\hat{\ld}_1(\pmb\omega,\mathbf{A})\le \cdots \le \hat{\ld}_{k-m}(\pmb\omega,\mathbf{A})$,
have the following properties.
 \begin{itemize}
\item [$(1)$] Let  $E_i=\{\hat{\ld}_{i}(\pmb\omega,\mathbf{A}):\,(\pmb\omega,\mathbf{A})\in \mathcal{O}\}$. Then for all $1\le i\le k-m$,
\begin{equation*}
 {\rm either}\, E_i\subset(-\infty,c_1)\; {\rm or}\;
 E_i\subset (c_2, +\infty),
\end{equation*}
and there exists $1\le i_0\le k$ such that $\hat{\ld}_i=\ld_{i_0}$
is continuous on $\mathcal{O}$.
\item [$(2)$] Let $(\pmb\omega_0,\mathbf{A}_0)\in \bar{\mathcal{O}}\setminus \mathcal{O}$, $\sharp_1(\sigma(\pmb\omega_0,\mathbf{A}_0))
=m$, and  $\sigma(\pmb\omega_0,\mathbf{A}_0)\subset(c_1, c_2)$.
If $ E_{i_0}\subset (-\infty,c_1)$ for some $1 \le i_0\le k-m$, then
\begin{equation*}
 \lim_{\mathcal{O}\ni(\pmb\omega,\mathbf{A})\to (\pmb\omega_0,\mathbf{A}_0)}\hat{\ld}_{i}(\pmb\omega,\mathbf{A})=-\infty,\;1\leq i\leq i_0.
\end{equation*}
If $ E_{j_0}\subset (c_2, +\infty)$ for some $1 \le j_0\le k-m$, then
\begin{equation*}
 \lim_{\mathcal{O}\ni(\pmb\omega,\mathbf{A})\to (\pmb\omega_0,\mathbf{A}_0)}\hat{\ld}_{j}(\pmb\omega,\mathbf{A})=+\infty,\;j_0\leq j\leq k-m.
\end{equation*}
\end{itemize}
\end{lemma}

The following result indicates that the monotonicity of $\lambda_n$ in a certain direction determines its asymptotic behavior in this direction.

\begin{lemma}\label{monotonicity-to-dis}
 Let $\mathcal{O}=\{(\pmb\omega,\mathbf{A})_{\nu}\in \Omega_N^\mathbb{C}\times \mathcal{B}^{\C}: \nu\in (\nu_0-\epsilon,\nu_0+\epsilon)\}$, where $(\pmb\omega,\mathbf{A})_{\nu}$ is continuously dependent on  $\nu\in(\nu_0-\epsilon,\nu_0+\epsilon)$ for
some $\epsilon>0$.
Assume that $\sharp_1(\sigma(\pmb\omega,\mathbf{A})_{\nu_0})=m \ge 0$, and for all
$\nu\in (\nu_0-\epsilon,\nu_0+\epsilon)\setminus \{\nu_0\}$, $\sharp_1(\sigma(\pmb\omega,\mathbf{A})_{\nu})=k>m$.
\begin{itemize}
\item [$(1)$] If  $\ld_n(\nu):=\ld_n((\pmb\omega,\mathbf{A})_{\nu})$ is non-increasing on $(\nu_0-\epsilon,\nu_0)$
for all $1\le n\le k$, then
\begin{align*}
\lim_{\nu\to \nu_0^-}\ld_n(\nu)=-\infty, \quad 1\le n\le k-m, \quad\lim_{\nu\to \nu_0^-}\ld_n(\nu)=\ld_{n-(k-m)}(\nu_0), \quad k-m+1\le n\le k.
\end{align*}
\item [$(2)$] If $\ld_n(\nu)$ is non-decreasing on $(\nu_0-\epsilon,\nu_0)$
for all $1\le n\le k$, then
\begin{align*}
  \lim_{\nu\to \nu_0^-}\ld_n(\nu)=\ld_{n}(\nu_0), \quad 1\le n\le m,\quad  \lim_{\nu\to \nu_0^-}\ld_n(\nu)=+\infty, \quad m+1\le n\le k.
\end{align*}
\item [$(3)$] If  $\ld_n(\nu)$ is non-increasing on $(\nu_0,\nu_0+\epsilon)$
for all $1\le n\le k$, then
\begin{align*}
\lim_{\nu\to \nu_0^+}\ld_n(\nu)=\ld_{n}(\nu_0), \quad 1\le n\le m, \quad \lim_{\nu\to \nu_0^+}\ld_n(\nu)=+\infty, \quad m+1\le n\le k.
\end{align*}
\item [$(4)$] If  $\ld_n(\nu)$ is non-decreasing on $(\nu_0,\nu_0+\epsilon)$
for all $1\le n\le k$, then
\begin{align*}
\lim_{\nu\to \nu_0^+}\ld_n(\nu)=-\infty, \quad 1\le n\le k-m,\quad \lim_{\nu\to \nu_0^+}\ld_n(\nu)=\ld_{n-(k-m)}(\nu_0), \quad k-m+1\le n\le k.
\end{align*}
\end{itemize}
\end{lemma}
 Note that the analyses in the proofs of Lemmas \ref{continuity criteria}, \ref{integer} and \ref{monotonicity-to-dis} are independent of the dimension of the Sturm-Liouville problem (\ref{Sturm-Liouville equation})--(\ref{boundary condition }). Thus they are indeed  a straightforward generalization of  Theorems 2.1, 2.3 and Lemma 2.7 in \cite{Zhu2}.

\section{Singularity of the $n$-th eigenvalue of discrete Sturm-Liouville problems}

In this section, we
completely characterize  singularity of the $n$-th eigenvalue on the boundary conditions for a fixed equation.
Then we characterize  singularity of the $n$-th eigenvalue on the equations for a fixed boundary condition under a non-degenerate assumption.

\subsection{Singularity  of the $n$-th eigenvalue on the boundary conditions}
Fix a Sturm-Liouville equation $\pmb\omega=(P^{-1},Q, W)$ such that $P_0^{-1}\in\mathcal{P}_d(\mathbb{C})$ in this subsection.
Let $K\subseteq \{1,2,\cdots,2d\}$.
For any boundary condition $\A=[A\;|\;B]\in \mathcal{O}_K^{\C}$, it follows from (\ref{co-sys}) that there exists
$S\in \mathcal{H}_{2d}(\C)$ such that $ \A=[(S\,|\, I_{2d})E_{K}]$. Let us  write $S=S(\mathbf{A})$ in the
partitioned form:
\begin{equation*}
  S(\mathbf{A})=\left(
      \begin{array}{cc}
        S_1 & S_2 \\
        S_2^* & S_3 \\
      \end{array}
    \right),
\end{equation*}
where $S_1, S_3\in \mathcal{H}_d(\C)$ and $S_2\in \mathcal{M}_d(\C)$.
Then it follows that
\begin{align*}
 (A,\,B) &=(S,\,I_{2d})E_{K}=\left(
                                \begin{array}{cccc}
                                  S_1 & S_2 & I_d & 0 \\
                                  S_2^*& S_3 & 0 & I_d \\
                                \end{array}
                              \right)\left(
          \begin{array}{cccccccc}
            E_{1} & 0 & I_d-E_{1} & 0 \\
            0 &  E_{2}& 0 &  I_d-E_{2}  \\
            E_{1}-I_d & 0 &  E_{1} & 0  \\
            0 &  E_{2}-I_d & 0 &  E_{2} \\
          \end{array}
        \right)\\
  &=\left(
      \begin{array}{cccc}
        S_1E_1+E_1 -I_d&S_2E_2 & S_1(I_d-E_1)+E_1 & S_2(I_d-E_2) \\
        S_2^*E_1 &S_3E_2+E_2-I_d & S_2^*(I_d-E_1) & S_3(I_d-E_2)+E_2 \\
      \end{array}
    \right).
\end{align*}
Recall that  $A_j, B_j\in\mathcal{M}_{2d\times d}$, $j=1,2$, are defined in (\ref{A1A2B1B2}).
  Then we have
\begin{align}
  &(A_1P_0^{-1}+B_1,\,B_2)\nonumber\\\label{matrix-trans-number-eigenvalue}
  =&\left(
                          \begin{array}{cc}
                            S_1(E_1P_0^{-1}+I_d-E_1)+(E_1-I_d)P_0^{-1} +E_1& S_2(I_d-E_2) \\
                           S_2^*(E_1P_0^{-1}+I_d-E_1) & S_3(I_d-E_2)+E_2 \\
                          \end{array}
                        \right).
\end{align}
From the structure of $E_1$ and the fact that $P_0^{-1}\in\mathcal{P}_{d}(\mathbb{C})$, we infer that
 $E_1P_0^{-1}+I_d-E_1$ is invertible.
Then it follows that
\begin{align}\label{matrix-operation}
& (A_1P_0^{-1}+B_1,B_2)\left(
                         \begin{array}{cc}
                           (E_1P_0^{-1}+I_d-E_1)^{-1} & 0 \\
                           0 & I_d \\
                         \end{array}
                       \right)\\\nonumber
  =&\left(
                          \begin{array}{cc}
                            S_1+[(E_1-I_d)P_0^{-1} +E_1](E_1P_0^{-1}+I_d-E_1)^{-1} & S_2(I_d-E_2) \\
                           S_2^* & S_3(I_d-E_2)+E_2 \\
                          \end{array}
                        \right).
\end{align}
Recall that $K_2$ is defined in (\ref{K1K2-definition}), and $e_i$  is  the $i$-th column of    $I_d$.
For convenience, we set $r=\sharp (K_2)$. Let us write $K_2=\{k_1+d,k_2+d,\cdots,k_r+d\}$
with $1\le k_1<k_2<\cdots <k_r\le d$,
and
\begin{align}\label{E0}
E_0=(e_{k_1},\cdots, e_{k_r}),
\end{align}
if $K_2\neq\emptyset$. For any  $\mathbf{A}\in\mathcal{B}^\mathbb{C}$,
we define
\begin{align}
 D(\mathbf{A})=(A_1P_0^{-1}+B_1,B_2).
\end{align}
For any  $\mathbf{A}\in\mathcal{O}_K^\mathbb{C}$, we define
\begin{align}\label{Hermite}
S^D_K(\mathbf{A})&=\left(
                          \begin{array}{cc}
                            S_1+[(E_1-I_d)P_0^{-1} +E_1](E_1P_0^{-1}+I_d-E_1)^{-1} & S_2E_0 \\
                           E_0^*S_2^* & E_0^*S_3E_0 \\
                          \end{array}
                        \right){\rm\;\;if \;\;}K_2\neq\emptyset,\\\label{Hermite2}
                        S^D_K(\mathbf{A})&= S_1+[(E_1-I_d)P_0^{-1} +E_1](E_1P_0^{-1}+I_d-E_1)^{-1}{\rm\;\;if\;\;}K_2=\emptyset.
\end{align}
Then the following result holds.
\begin{lemma} Let  $\mathbf{A}\in\mathcal{O}_K^\mathbb{C}$. Then
\begin{align}\label{rank}
 {\rm rank}\, D(\mathbf{A})= {\rm rank}\,S^D_K(\mathbf{A})+d-r,
\end{align}
and $S^D_K(\mathbf{A})\in\mathcal{H}_{d+r}(\mathbb{C})$.
\end{lemma}

\begin{proof}
By (\ref{E1E2}), ${\rm rank} (E_2)=d-r$ and thus  (\ref{rank}) holds.
To prove $S^D_K(\mathbf{A})\in\mathcal{H}_{d+r}(\mathbb{C})$, it suffices to show that
\begin{align}\label{E1idp0hd}
[(E_1-I_d)P_0^{-1} +E_1](E_1P_0^{-1}+I_d-E_1)^{-1}\in\mathcal{H}_{d}(\mathbb{C}).
\end{align}
Direct computation gives
\begin{align*}
 (P_0^{-1}E_1+I_d-E_1) [(E_1-I_d)P_0^{-1} +E_1]=[P_0^{-1}(E_1-I_d) +E_1](E_1P_0^{-1}+I_d-E_1).
\end{align*}
Since  $(P_0^{-1}E_1+I_d-E_1)$ and $ (E_1P_0^{-1}+I_d-E_1)$ are invertible, we have
\begin{align*}
  [(E_1-I_d)P_0^{-1} +E_1](E_1P_0^{-1}+I_d-E_1)^{-1}=(P_0^{-1}E_1+I_d-E_1)^{-1}[P_0^{-1}(E_1-I_d) +E_1],
\end{align*}
which yields (\ref{E1idp0hd}).
\end{proof}
Define
\begin{align}\label{layer}
\mathcal{B}_{k}:=&\{\mathbf{A}\in\mathcal{B}^\mathbb{C}\mid r^0( D(\mathbf{A}))=k\}, \; 0\leq k\leq 2d, \\\label{area division}
\mathcal{B}_K^{(r^0,r^+,r^-)} :=&\{\mathbf{A}\in\mathcal{O}_K^\mathbb{C}\mid r^0=r^0(S^D_K(\mathbf{A})), \,r^\pm=r^\pm(S^D_K(\mathbf{A}))\}
\end{align}
for nonnegative integers $r^0, r^\pm$ with $ r^0+ r^-+ r^+=d+r$. (\ref{layer}) gives the $2d+1$ layers in $\mathcal{B}^\mathbb{C}$, while
(\ref{area division}) divides $\mathcal{O}_K^\mathbb{C}$ into different areas.
Theorem \ref{main result discrete case} below indicates that the $n$-th eigenvalue exhibits the same singularity in any given area.
By Lemma \ref{number of eigenvalues}, we have the following result.

 \begin{lemma}\label{number eigenvalue Bk}
 $\sharp_1(\sigma(\pmb\omega,\mathbf{A}))=Nd-k$ for  $\mathbf{A}\in \mathcal{B}_k$,  and
$\sharp_1(\sigma(\pmb\omega,\mathbf{A}))= Nd-r^0$ for  $\mathbf{A}\in \mathcal{B}_{K}^{(r^0,r^+,r^-)}$.
\end{lemma}

\begin{lemma}\label{neighborhood}
Let  $\mathbf{A}\in \mathcal{B}_{K}^{(r_1^0,r_1^+,r_1^-)}$. Then
\begin{align*}
\mathcal{U}_\varepsilon^{(r^0,r^+,r^-)}:=\mathcal{U}_\varepsilon\cap \mathcal{B}_{K}^{(r^0,r^+,r^-)}
\end{align*}
with $\mathcal{U}_\varepsilon=\{\mathbf{B}\in \mathcal{O}_K^\mathbb{C}:\| S(\mathbf{B})- S(\mathbf{A})\|_{\mathcal{M}_{2d}}<\varepsilon\}$
is path connected for any $r^0\leq r_1^0$, $r^\pm\geq r_1^\pm$ satisfying $r^0+r^++r^-=d+r$, and $\varepsilon>0$ sufficiently small.
\end{lemma}
\begin{proof}
The proof is similar as Lemma 7.2 in \cite{Hu-Liu-Wu-Zhu2018}.
\end{proof}

We are now in a position to give the complete characterization of singularity of the $n$-th eigenvalue on the boundary conditions.

\begin{theorem}\label{main result discrete case} Fix $\pmb\omega\in\Omega_N^\mathbb{C}$.

(1) Let $0\leq k\leq 2d$. Then the restriction of $\lambda_n$ to $\mathcal{B}_k$ is continuous for any $1\leq n\leq Nd-k$.

(2) Consider the restriction of  $\lambda_n$ to $\mathcal{O}_K^\mathbb{C}$, where $K\subset \{1,\cdots,2d\}$.
Let $0\le r^0< r^0_1\leq d+r$ and $ r^\pm\geq r_1^\pm$.
Then  for any
 $\mathbf{A}\in \mathcal{B}_{K}^{(r_1^0,r_1^+,r_1^-)}$, we have
\begin{align}\label{both bigger1}
   & \lim_{\mathcal{B}_K^{(r^0,r^+,r^-)}\ni \mathbf{B}\to \mathbf{A}}\ld_{n}(\mathbf{B})=-\infty, \quad 1\le n\le r^+-r_1^+, \\
   \label{both bigger2}
    & \lim_{\mathcal{B}_K^{(r^0,r^+,r^-)}\ni\mathbf{B}\to\mathbf{A}}\ld_{n}(\mathbf{B})
    =\ld_{n-(r^+-r^+_1)}(\mathbf{A}),\quad r^+-r_1^+< n\le Nd-r^0-( r^--r_1^-),\\\label{both bigger3}
     & \lim_{\mathcal{B}_K^{(r^0,r^+,r^-)}\ni\mathbf{B}\to \mathbf{A}}\ld_{n}(\mathbf{B})=+\infty, \quad Nd-r^0-( r^--r_1^-)< n\le Nd-r^0.
\end{align}
Consequently, the singular set is $\cup_{1\leq k\leq 2d} \mathcal{B}_k$.
\end{theorem}

\begin{proof}
(1) is a direct consequence of Lemmas \ref{continuity criteria} and \ref{number eigenvalue Bk}.
Now, we prove (2). Note that $\sharp_1(\sigma({\pmb\omega},\mathbf{A}))=Nd-r_1^0$  by Lemma \ref{number eigenvalue Bk}.
Choose  $c_1,c_2\in\mathbb{R}$  such that $\sharp_1(\sigma(\pmb\omega,\mathbf{A})\cap(c_1,c_2))=Nd-r_1^0$.
 By Lemma \ref{small perturbation}, there exists  $\varepsilon>0$  such that for all  $\mathbf{B}\in\mathcal{U}_\varepsilon$ defined in Lemma \ref{neighborhood}, we have
 $\sharp_1(\sigma(\pmb\omega,\mathbf{B})\cap(c_1,c_2))=Nd-r_1^0$
 and $c_1,c_2\notin\sigma(\pmb\omega,\mathbf{B})$.
It follows from Lemma \ref{neighborhood} that $\mathcal{U}_\varepsilon^{(r^0,r^+,r^-)}$  is path connected.
By Lemma \ref{number eigenvalue Bk},  $\sharp_1(\sigma({\pmb\omega},\mathbf{B}))=Nd-r^0$ for  $\mathbf{B}\in\mathcal{U}_\varepsilon^{(r^0,r^+,r^-)}$, and thus
 $\sharp_1\left(\sigma({\pmb\omega},\mathbf{B})\cap((-\infty,c_1)\cup(c_2,+\infty))\right)=r_1^0-r^0$. Let
$\sigma({\pmb\omega},\mathbf{B})\cap((-\infty,c_1)\cup(c_2,+\infty)):=\{\hat \lambda_1(\mathbf{B})\leq\cdots\leq\hat \lambda_{r_1^0-r^0}(\mathbf{B})\}$ for   $\mathbf{B}\in\mathcal{U}_\varepsilon^{(r^0,r^+,r^-)}$.
By   Lemma \ref{integer}  (1), either $\hat \lambda_n(\mathcal{U}_\varepsilon^{(r^0,r^+,r^-)})\subset(-\infty,c_1)$ or $\hat \lambda_n(\mathcal{U}_\varepsilon^{(r^0,r^+,r^-)})\subset(c_2, +\infty)$ for all $1\leq n\leq r_1^0-r^0$.
 Then we divide our proof in two steps.



{\bf Step 1.} We show that
\begin{align}\label{r+bigger1}\lim\limits_{\mathcal{B}_K^{(r^0,r^+,r^-)}\ni\mathbf{B}\rightarrow\mathbf{A}}\lambda_{n}(\mathbf{B})=&-\infty,\;1\leq n\leq r^+-r_1^+,\\\label{r+bigger2}
\lim\limits_{\mathcal{B}_K^{(r^0,r^+,r^-)}\ni\mathbf{B}\rightarrow \mathbf{A}}\lambda_{n}(\mathbf{B})=&\lambda_{n-(r^+-r_1^+)}(\mathbf{A}),\;r^+-r_1^+< n \leq Nd-r^0,\end{align}
for  $r^+>r_1^+, r^-=r_1^-$; and
\begin{align}\label{r-bigger1}\lim\limits_{\mathcal{B}_K^{(r^0,r^+,r^-)}\ni\mathbf{B}\rightarrow \mathbf{A}}\lambda_{n}(\mathbf{B})&=\lambda_{n}(\mathbf{A}),\;1\leq  n \leq (Nd-r^0)-(r^--r_1^-),\\
\label{r-bigger2}
\lim\limits_{\mathcal{B}_K^{(r^0,r^+,r^-)}\ni\mathbf{B}\rightarrow\mathbf{A}}\lambda_{n}(\mathbf{B})&=+\infty,\;(Nd-r^0)-(r^--r_1^-)< n\leq Nd-r^0
\end{align}
for $ r^+=r_1^+, r^->r_1^-$.

Consider  $r^+>r_1^+$ and $r^-=r_1^-$. In this case, $r_1^0-r^0=r^+-r_1^+$.
Note that there exists a unitary matrix $M\in\mathcal{M}_{d+r}$ such that
$$S^D_K(\mathbf{A})=M\begin{pmatrix}M_+&&\\&M_-&\\&&0_{r_1^0}
\end{pmatrix}M^*,$$
where $M_+={\rm diag}\{\mu_1,\cdots,\mu_{r_1^+} \}$ with $\mu_{i}>0$, $1\leq i\leq r_1^+$, and $M_-={\rm diag}\{\nu_1,\cdots,\nu_{r_1^-} \}$ with $\nu_j<0$, $1\leq j\leq r_1^-$. Recall that $S=S(\mathbf{A})$. If $K_2\neq\emptyset$, we define
$\mathbf{B}_t=[S(\mathbf{B}_t)\,|\,I_{2d}]E_K$ with
\begin{align}\label{def-sij-B}
&(s_{ij}(\mathbf{B}_t))_{i,j\in\{1,\cdots,d,k_1+d,\cdots,k_r+d\}}\\\nonumber
=&\begin{pmatrix}S_1&S_2E_0\\ E_0^*S_2^*&E_0^* S_3E_0\end{pmatrix}
+M\begin{pmatrix}0_{r_1^++r_1^-}&&\\&tI_{r^+-r_1^+}&\\&&0_{r^0}
\end{pmatrix}M^*,\end{align}
$t\geq0$ is sufficiently small, and $s_{ij}(\mathbf{B}_t)=s_{ij}$ if $i\in\{d+1,\cdots,2d\}\setminus K_2$ or $j\in\{d+1,\cdots,2d\}\setminus K_2$. If $K_2=\emptyset$, we only modify (\ref{def-sij-B}) as
\begin{align*}
(s_{ij}(\mathbf{B}_t))_{i,j\in\{1,\cdots,d\}}=S_1
+M\begin{pmatrix}0_{r_1^++r_1^-}&&\\&tI_{r^+-r_1^+}&\\&&0_{r^0}
\end{pmatrix}M^*\end{align*}
in the definition of $\mathbf{B}_t$.
Then $\mathbf{B}_0=\mathbf{A}$,
$$S^D_K(\mathbf{B}_t)=M\begin{pmatrix}M_+&&&\\&M_-&&\\&&tI_{r^+-r_1^+}&\\&&&0_{r^0}
\end{pmatrix}M^*,t>0,$$
and thus $r^0(S^D_K(\mathbf{B}_t))=r^0$, $r^\pm(S^D_K(\mathbf{B}_t))=r^\pm$, which gives $\mathbf{B}_t\in\mathcal{U}_\varepsilon^{(r^0,r^+,r^-)}$.
 Moreover,
 $\sharp_1(\sigma({\pmb\omega},\mathbf{B}_t))$ $=Nd-r^0$ for $t>0$,  and  $\sharp_1(\sigma({\pmb\omega},\mathbf{B}_0))=Nd-r_1^0$.
 It follows from Lemma \ref{continuity criteria} that for any fixed $1\leq n\leq Nd-r^0$, $\lambda_n(\mathbf{B}_{\cdot})$ is locally a continuous eigenvalue branch for $t>0$. Since
 $$S^D_K(\mathbf{B}_{t_2})-S^D_K(\mathbf{B}_{t_1})=M\begin{pmatrix}0_{r_1^++r_1^-}&&\\&(t_2-t_1)I_{r^+-r_1^+}&\\&&0_{r^0}
\end{pmatrix}M^*$$
is a positive semi-definite matrix, we get by Corollary \ref{cor-monotonicity-boundary} that
 $\lambda_n(\mathbf{B}_{t_1})\leq \lambda_n(\mathbf{B}_{t_2})$ with $0<t_1<t_2$ for all $1\leq n \leq Nd-r^0$. Hence, by   Lemma \ref{monotonicity-to-dis} (4),
$\lim_{t\rightarrow {0}^+}\lambda_{n}({\mathbf{B}}_t)=-\infty, 1\leq n\leq r^+-r_1^+.$
Therefore,
there exists $t_0>0$ such that  $\mathbf{B}_{t_0}\in\mathcal{U}_\varepsilon^{(r^0,r^+,r^-)}$ and  ${\lambda}_{n}(\mathbf{B}_{t_0})<c_1$, $1\leq n\leq r^+-r_1^+$, which yields that  $\hat{\lambda}_{n}(\mathbf{B}_{t_0})={\lambda}_{n}(\mathbf{B}_{t_0})$. According to Lemma \ref{integer} (1),
 $\hat \lambda_n(\mathcal{U}_\varepsilon^{(r^0,r^+,r^-)})={\lambda}_{n}(\mathcal{U}_\varepsilon^{(r^0,r^+,r^-)})\subset(-\infty,c_1)$, $1\leq n\leq r^+-r_1^+$,
and (\ref{r+bigger1}) holds.
Thanks to Lemma \ref{infinity-finity}, we get (\ref{r+bigger2}).

Consider  $r^+=r_1^+$ and $r^->r_1^-$.
Since (\ref{r-bigger1})--(\ref{r-bigger2}) can be shown in a similar way, we omit the details.

{\bf Step 2.} Show that
(\ref{both bigger1})--(\ref{both bigger3}) hold
for $r^\pm>r_1^\pm$.

In this case, $r_1^0-r^0=(r^+-r_1^+)+(r^--r_1^-)$.
It follows from (\ref{r+bigger1}) and (\ref{r-bigger2}) that
\begin{align*}
\lim\limits_{\mathcal{B}_K^{(r_1^0-(r^+-r_1^+),r^+,r_1^-)}\ni\mathbf{B}\rightarrow\mathbf{A}}\lambda_{n}(\mathbf{B})=&-\infty,\;1\leq n\leq r^+-r_1^+,\\
\lim\limits_{\mathcal{B}_K^{(r_1^0-(r^--r_1^-),r_1^+,r^-)}\ni\mathbf{B}\rightarrow\mathbf{A}}\lambda_{n}(\mathbf{B})=&+\infty,\;Nd-r_1^0< n\leq Nd-r_1^0+(r^--r_1^-).
\end{align*}
This implies that  $\lambda_n(\tilde{\mathbf{B}}_1) \in(-\infty,c_1)$ with $1\leq n\leq r^+-r_1^+$,   and $\lambda_{n}(\tilde{\mathbf{B}}_2)\in(c_2,+\infty)$ with $Nd-r_1^0< n\leq Nd-r_1^0+(r^--r_1^-)$ for any fixed
$\tilde{\mathbf{B}}_1\in \mathcal{U}_\varepsilon^{(r_1^0-(r^+-r_1^+),r^+,r_1^-)}$ and $\tilde{\mathbf{B}}_2\in \mathcal{U}_\varepsilon^{(r_1^0-(r^--r_1^-),r_1^+,r^-)}$.

Note that $(r^0,r^+,r^-)=(r_1^0-(r^+-r_1^+)-(r^--r^-_1),r^+,r_1^-+(r^--r_1^-))$.
Then we infer from (\ref{r-bigger1}) that
\begin{align*}
\lim\limits_{\mathcal{B}_K^{(r^0,r^+,r^-)}\ni\mathbf{B}\rightarrow\tilde{\mathbf{B}}_1}\lambda_{n}(\mathbf{B})=&\lambda_n(\tilde{\mathbf{B}}_1),\;1\leq n\leq Nd-r^0-(r^--r_1^-).
\end{align*}
Since $N\geq 2$, we get that
\begin{align}\label{N2}
Nd-r^0-(r^--r_1^-)\geq r^+-r_1^+.
\end{align}
Therefore, there exists
$\tilde{\mathbf{B}}_3\in \mathcal{U}_\varepsilon^{(r^0,r^+,r^-)}$ such that   $\lambda_n(\tilde{\mathbf{B}}_3) \in(-\infty,c_1)$ with $1\leq n\leq r^+-r_1^+$.

On the other hand, $(r^0,r^+,r^-)=(r_1^0-(r^+-r_1^+)-(r^--r_1^-),r_1^++(r^+-r_1^+),r^-)$. Thus we get by (\ref{r+bigger2}) that
\begin{align*}
\lim\limits_{\mathcal{B}_K^{(r^0,r^+,r^-)}\ni\mathbf{B}\rightarrow\tilde{\mathbf{B}}_2}\lambda_{n}(\mathbf{B})
=&\lambda_{n-(r^+-r_1^+)}(\tilde{\mathbf{B}}_2),\;r^+-r_1^+< n\leq Nd-r^0,
\end{align*}
which, along with (\ref{N2}), yields that  there exists
$\tilde{\mathbf{B}}_4\in \mathcal{U}_\varepsilon^{(r^0,r^+,r^-)}$ such that
 $\lambda_n(\tilde{\mathbf{B}}_4) \in(c_2,+\infty)$ with $Nd-r^0-(r^--r_1^-)<n\leq Nd-r^0$.
Therefore, we have shown
\begin{align}\label{B3B4}
\sharp_1(\sigma(\pmb\omega,\tilde{\mathbf{B}}_3)\cap(-\infty,c_1))=r^+-r_1^+,\;\;
\sharp_1(\sigma(\pmb\omega,\tilde{\mathbf{B}}_4)\cap(c_2,+\infty))=r^--r_1^-.
\end{align}
 Note that $\mathcal{U}_\varepsilon^{(r^0,r^+,r^-)}$  is path connected and $\sharp_1(\sigma(\pmb\omega,{\mathbf{B}})\cap(c_1,c_2))=Nd-r_1^0=Nd-r^0-(r^+-r_1^+)-(r^--r_1^-)$ for all $\mathbf{B}\in \mathcal{U}_\varepsilon^{(r^0,r^+,r^-)}$. Thus we infer from (\ref{B3B4}) and Lemma \ref{integer} (1) that $\lambda_n(\mathcal{U}_\varepsilon^{(r^0,r^+,r^-)})\subset (-\infty,c_1)$ for all $1\leq n\leq r^+-r_1^+$, and
$\lambda_{n}(\mathcal{U}_\varepsilon^{(r^0,r^+,r^-)})\subset(c_2,+\infty) $ for all  $Nd-r^0-(r^--r_1^-)<n\leq Nd-r^0$.
Then it follows from Lemma \ref{integer} (2) that   (\ref{both bigger1}) and (\ref{both bigger3}) hold. This, along with Lemma \ref{infinity-finity}, implies that
 (\ref{both bigger2}) holds.
This  completes the proof.
\end{proof}

\subsection{Singularity of the $n$-th eigenvalue on the Sturm-Liouville equations} Fix a boundary condition $\mathbf{A}=[(A_1,\;A_2)|(B_1,\;B_2)]=[S\;|\;I_{2d}]E_K\in\mathcal{O}_K^\mathbb{C}$.
In this subsection, we always assume that one of the following non-degenerate conditions holds:
\begin{align}\label{non-degenerate assumption1}
& R \;\;{\rm and}\;\;  E_0^*(S_2^*E_1R^{-1}S_2-S_3)E_0\;\;{\rm are\; invertible\;if} \;K_2\neq\emptyset,\\\label{non-degenerate assumption2}
 & R\;{\rm is\; invertible\;if} \;K_2=\emptyset,
\end{align}
  where $E_0$ is defined in (\ref{E0}) and
$R:=S_1E_1+E_1-I_d$. In particular, the assumption holds for any $\mathbf{A}\in\mathcal{O}_K^\mathbb{C}$ when $K_1=\{1,\cdots,d\}$ and $K_2=\emptyset$.
 For any $\pmb\omega=(P^{-1},Q,W)\in\Omega_N^\mathbb{C}$, we have by (\ref{matrix-trans-number-eigenvalue}) that
\begin{align*}
  (A_1P_0^{-1}+B_1,\,B_2)=\left(
                          \begin{array}{cc}
                            RP_0^{-1} +S_1(I_d-E_1)+E_1& S_2(I_d-E_2) \\
                           S_2^*(E_1P_0^{-1}+I_d-E_1) & S_3(I_d-E_2)+E_2 \\
                          \end{array}
                        \right).
\end{align*}
Then
 \begin{align}\label{equation transform1}
 & \begin{pmatrix} R^{-1}&0\\S_2^*E_1R^{-1}&-I_d\end{pmatrix}(A_1P_0^{-1}+B_1,\,B_2)\\
 =&     \begin{pmatrix}
                           P_0^{-1}+ R^{-1}(S_1(I_d-E_1)+E_1)& R^{-1}S_2(I_d-E_2) \\
                           S_2^*[E_1-I_d+E_1R^{-1}(S_1(I_d-E_1)+E_1)] & (S_2^*E_1R^{-1}S_2-S_3)(I_d-E_2)-E_2
                          \end{pmatrix}.\nonumber
\end{align}
Next, we analyze the partitioned structure of the matrix above.
\begin{lemma}\label{matrix-iden}
$[E_1-I_d+E_1R^{-1}(S_1(I_d-E_1)+E_1)]R^*=I_d$ and $ S_2^*E_1R^{-1}S_2-S_3\in\mathcal{H}_d(\mathbb{C})$.
\end{lemma}

\begin{proof}
Direct computation gives
\begin{align*}
[E_1-I_d+E_1R^{-1}(S_1(I_d-E_1)+E_1)]R^*=I_d-E_1+E_1R^{-1}(E_1S_1-S_1+S_1E_1).
\end{align*}
Since \begin{align}\label{iden}E_1R^{-1}=(E_1S_1+E_1-I_d)^{-1}E_1,\end{align} we have
\begin{align*}
&I_d-E_1+E_1R^{-1}(E_1S_1-S_1+S_1E_1)\\
=&I_d-E_1+(E_1S_1+E_1-I_d)^{-1}E_1(E_1S_1-S_1+S_1E_1)\\
=&I_d-E_1+(E_1S_1+E_1-I_d)^{-1}E_1S_1E_1\\
=&I_d-E_1+(E_1S_1+E_1-I_d)^{-1}(E_1S_1+E_1-I_d)E_1=I_d.\\
\end{align*}
$S_2^*E_1R^{-1}S_2-S_3\in\mathcal{H}_d(\mathbb{C})$ follows directly from (\ref{iden}).
\end{proof}
By Lemma \ref{matrix-iden}, we have $S_2^*[E_1-I_d+E_1R^{-1}(S_1(I_d-E_1)+E_1)]=(R^{-1}S_2)^*$. In the case that
$K_2\neq\emptyset$,
we define
$$\hat T_1:=R^{-1}S_2E_0,\;\hat T_2:= E_0^*(S_2^*E_1R^{-1}S_2-S_3)E_0.$$
Then by the assumption (\ref{non-degenerate assumption1}) and Lemma \ref{matrix-iden}, $\hat T_2\in\mathcal{H}_r(\mathbb{C})$ is invertible.  Direct computation implies that
\begin{align}\label{equation transform2}
&
\begin{pmatrix}
                           I_d&-\hat T_1\hat T_2^{-1}\\
                          0 & \hat T_2^{-1}
                          \end{pmatrix}\begin{pmatrix}
                           P_0^{-1}+ R^{-1}(S_1(I_d-E_1)+E_1)&\hat T_1 \\
                          \hat T_1^* & \hat T_2
                          \end{pmatrix}
                          \begin{pmatrix}
                           I_d&0\\
                          -\hat T_2^{-1}\hat T_1^* & I_r
                          \end{pmatrix}\\
 =&                         \begin{pmatrix}
                           P_0^{-1}+ R^{-1}(S_1(I_d-E_1)+E_1)-\hat T_1\hat T_2^{-1}\hat T_1^*&0 \\
                          0 &  I_r
                          \end{pmatrix}.\nonumber
\end{align}
In the case that $K_2=\emptyset$, we have $E_2=I_d$. Then the next transformation after (\ref{equation transform1}) is
\begin{align}\label{transformation2222222222}
&\begin{pmatrix}
                           P_0^{-1}+ R^{-1}(S_1(I_d-E_1)+E_1)& 0 \\
                           \tilde R & -I_d
                          \end{pmatrix}\begin{pmatrix}
                           I_d& 0 \\
                           \tilde R & -I_d
                          \end{pmatrix}\\\nonumber
                          =&\begin{pmatrix}
                           P_0^{-1}+ R^{-1}(S_1(I_d-E_1)+E_1)& 0 \\
                           0 & I_d
                          \end{pmatrix},
\end{align}
where $\tilde R= S_2^*[E_1-I_d+E_1R^{-1}(S_1(I_d-E_1)+E_1)]$.
For any  $\pmb\omega\in\Omega_N^\mathbb{C}$, we define
\begin{align}\label{Tpmbomega-def1}
T(\pmb\omega)&:=     P_0^{-1}+ R^{-1}(S_1(I_d-E_1)+E_1)-\hat T_1\hat T_2^{-1}\hat T_1^* {\rm\;\;if\;\;}K_2{\neq\emptyset},\\\label{Tpmbomega-def2}
T(\pmb\omega)&:=     P_0^{-1}+ R^{-1}(S_1(I_d-E_1)+E_1){\rm\;\;if\;\;}K_2{=\emptyset},
\end{align}
and $F(\pmb\omega):=(A_1P_0^{-1}+B_1,\,B_2)$  in both cases.
Then
$ {\rm rank}\, F(\pmb\omega)= {\rm rank}\,T(\pmb\omega)+d.$
Moreover, we have the following result.
\begin{lemma} Let  $\pmb\omega\in\Omega_N^\mathbb{C}$. Then $T(\pmb\omega)\in\mathcal{H}_{d}(\mathbb{C})$.
\end{lemma}
\begin{proof}
Since
\begin{align*}(S_1-S_1E_1+E_1)R^*=&S_1E_1+E_1S_1-S_1
=R(S_1-E_1S_1+E_1),\end{align*}
and $R$ is invertible, we have
$$R^{-1}(S_1-S_1E_1+E_1)=(S_1-E_1S_1+E_1)(R^*)^{-1}.$$
This implies that $ P_0^{-1}+ R^{-1}(S_1(I_d-E_1)+E_1)\in\mathcal{H}_{d}(\mathbb{C})$.
Since $\hat T_2\in\mathcal{H}_r(\mathbb{C})$ when $K_2\neq\emptyset$, we get $T(\pmb\omega)\in \mathcal{H}_{d}(\mathbb{C})$.\end{proof}
Let
\begin{align*}
l_1=\max_{\pmb\omega\in\Omega_N^\mathbb{C}}r^0(F(\pmb\omega)).
\end{align*}
Then $l_1\leq d$. Define
\begin{align}
\mathcal{E}_{k}:=&\{\pmb\omega\in\Omega_N^\mathbb{C}\mid r^0( F(\pmb\omega))=k\}, \; 0\leq k\leq l_1,\label{layer-equa} \\
\mathcal{E}^{(r^0,r^+,r^-)} :=&\{\pmb\omega\in\Omega_N^\mathbb{C}\mid r^0=r^0(T(\pmb\omega)), \,r^\pm=r^\pm(T(\pmb\omega))\},\label{layer-area}
\end{align}
for attainable  nonnegative integers $r^0, r^+, r^-$ with $ r^0+ r^++ r^-=d$. The $l_1+1$ layers of $\Omega_N^\mathbb{C}$ are given in (\ref{layer-equa}), while the areas' division is provided in (\ref{layer-area}). Note here that not all the  nonnegative integers $0\leq r^0,r^\pm\leq d$ satisfying $r^0+r^++r^-=d$ can be achievable in general, since $P_0^{-1}\in \mathcal{P}_d(\mathbb{C})$ while it is not necessary that   $P_0^{-1}-T(\pmb\omega)\in \mathcal{P}_d(\mathbb{C})$.
Similarly, it is possible that $l_1<d$. The following result is a direct consequence of Lemma \ref{number of eigenvalues}.
\begin{lemma}\label{enumber of eigenvalue}
 $\sharp_1(\sigma(\pmb\omega,\mathbf{A}))=Nd-k$ for any $\pmb\omega\in \mathcal{E}_k$,  and
$\sharp_1(\sigma(\pmb\omega,\mathbf{A}))= Nd-r^0$ for any $\pmb\omega\in \mathcal{E}^{(r^0,r^+,r^-)}$.
\end{lemma}
Note that the transformations (\ref{equation transform1}), (\ref{equation transform2}) and (\ref{transformation2222222222}) are independent of $\pmb\omega\in\Omega_N^\mathbb{C}$.
Moreover, the following result holds by the construction of $T(\pmb\omega)$ and a similar argument as that in the proof of Lemma 7.2 in \cite{Hu-Liu-Wu-Zhu2018}.
\begin{lemma}\label{equation neighborhood}
Let $\pmb\omega\in \mathcal{E}^{(r_1^0,r_1^+,r_1^-)}$. Then
$\mathcal{V}_\varepsilon^{(r^0,r^+,r^-)}:=\{\pmb\sigma\in\Omega_N^\mathbb{C}:\| \pmb\sigma- \pmb\omega\|_{\C^{(3N+1)d^2}}<\varepsilon\}\cap \mathcal{E}^{(r^0,r^+,r^-)}$
is path connected for any $0\leq r^0\leq r_1^0$, $r^\pm\geq r_1^\pm$  satisfying $r^0+r^++r^-=d$, and $\varepsilon>0$ sufficiently small.
\end{lemma}

\begin{theorem}\label{main result equation discrete case} Fix   $\mathbf{A}\in\mathcal{O}_K^\mathbb{C}$.

(1) Let $0\leq k\leq l_1$. Then the restriction of $\lambda_n$ to $\mathcal{E}_k$ is continuous for any $1\leq n\leq Nd-k$.

(2) Consider the restriction of  $\lambda_n$ to $\Omega_N^\mathbb{C}$.
Let $0\le r^0< r^0_1\leq l_1$ and $ r^\pm\geq r_1^\pm$.
Then  for any
 $\pmb\omega\in \mathcal{E}^{(r_1^0,r_1^+,r_1^-)}$, we have
\begin{align}\label{eboth bigger1}
   & \lim_{\mathcal{E}^{(r^0,r^+,r^-)}\ni \pmb{\sigma}\to \pmb{\omega}}\ld_{n}(\pmb\sigma)=-\infty, \quad 1\le n\le r^--r_1^-, \\
   \label{eboth bigger2} & \lim_{\mathcal{E}^{(r^0,r^+,r^-)}\ni\pmb{\sigma}\to \pmb{\omega}}\ld_{n}(\pmb\sigma)
    =\ld_{n-(r^--r^-_1)}(\pmb\omega), \quad r^--r_1^-< n\le Nd-r^0-( r^+-r_1^+),\\
   \label{eboth bigger3}  & \lim_{\mathcal{E}^{(r^0,r^+,r^-)}\ni\pmb{\sigma}\to \pmb{\omega}}\ld_{n}(\pmb{\sigma})=+\infty, \quad Nd-r^0-( r^+-r_1^+)< n\le Nd-r^0.
\end{align}
Consequently, the singular set is $\cup_{1\leq k\leq l_1}\mathcal{E}_k$.
\end{theorem}

\begin{proof}
By Lemma \ref{enumber of eigenvalue}, $\sharp_1(\sigma(\pmb\omega,\mathbf{A}))=Nd-k$ for any $\pmb\omega\in \mathcal{E}_k$.
It follows from Lemma \ref{continuity criteria} that (1) holds.
Choose  $c_1,c_2\in\mathbb{R}$  such that $\sharp_1(\sigma(\pmb\omega,\mathbf{A})\cap(c_1,c_2))=Nd-r_1^0$.
Then Lemmas \ref{small perturbation}, \ref{integer}  (1)   and \ref{equation neighborhood} ensure that  for  $\pmb\sigma\in\mathcal{V}_\varepsilon^{(r^0,r^+,r^-)}$ with $\varepsilon>0$ small enough, all the eigenvalues of  $(\pmb\sigma,\mathbf{A})$ outside $[c_1,c_2]$, denoted by $\tilde \lambda_1(\pmb\sigma)\leq\cdots\leq\tilde \lambda_{r_1^0-r^0}(\pmb\sigma)$, satisfy that
either $\tilde \lambda_n(\mathcal{V}_\varepsilon^{(r^0,r^+,r^-)})\subset (-\infty,c_1)$ or $\tilde \lambda_n(\mathcal{V}_\varepsilon^{(r^0,r^+,r^-)})\subset(c_2, +\infty)$, $1\leq n\leq r_1^0-r^0$.
 Then we divide our proof in two steps.



{\bf Step 1.} We show that
\begin{align}\label{er-bigger1}\lim\limits_{\mathcal{E}^{(r^0,r^+,r^-)}\ni\pmb\sigma\rightarrow\pmb\omega}\lambda_{n}(\pmb\sigma)=&-\infty,\;1\leq n\leq r^--r_1^-,\\\label{er-bigger2}
\lim\limits_{\mathcal{E}^{(r^0,r^+,r^-)}\ni\pmb\sigma\rightarrow \pmb\omega}\lambda_{n}(\pmb\sigma)=&\lambda_{n-(r^--r_1^-)}(\pmb\omega),\;r^--r_1^-< n \leq Nd-r^0\end{align}
for  $r^->r_1^-, r^+=r_1^+$; and
\begin{align}\label{er+bigger1}\lim\limits_{\mathcal{E}^{(r^0,r^+,r^-)}\ni\pmb\sigma\rightarrow \pmb\omega}\lambda_{n}(\pmb\sigma)&=\lambda_{n}(\pmb\omega),\;1\leq  n \leq (Nd-r^0)-(r^+-r_1^+),\\
\label{er+bigger2}
\lim\limits_{\mathcal{E}^{(r^0,r^+,r^-)}\ni\pmb\sigma\rightarrow\pmb\omega}\lambda_{n}(\pmb\sigma)&=+\infty,\;(Nd-r^0)-(r^+-r_1^+)< n\leq Nd-r^0
\end{align}
for $ r^-=r_1^-, r^+>r_1^+$.

We only prove (\ref{er-bigger1})--(\ref{er-bigger2}), since (\ref{er+bigger1})--(\ref{er+bigger2}) can be proved similarly.
Let $L\in\mathcal{M}_{d}(\mathbb{C})$ be a unitary matrix such that
$T(\pmb\omega)=L\,{\rm diag}\{\tilde\mu_1,\cdots,\tilde\mu_{d} \}L^*,$
where $\tilde\mu_{i}$, $1\leq i\leq d$, are the eigenvalues of $T(\pmb\omega)$ and $\tilde \mu_1=\cdots=\tilde \mu_{r_1^0}=0$. Recall that
 $P_0^{-1}(\pmb\omega)$ is used to indicate its dependence on $\pmb\omega$, while all the  components of $\pmb\omega$ except  $P_0^{-1}$ are fixed.  Define
$$P_0^{-1}(\pmb\sigma_t)=P_0^{-1}(\pmb\omega)+L\begin{pmatrix}tI_{r_1^0-r^0}&\\&0_{d-(r_1^0-r^0)}
\end{pmatrix}L^*$$
 with $t\leq 0$ small enough.
Then $\pmb\sigma_0=\pmb\omega$ and
  $\pmb{\sigma}_t\in\mathcal{V}_\varepsilon^{(r^0,r^+,r^-)}$, $t<0$.
  Since
 $P_0^{-1}(\pmb\sigma_{t_2})-P_0^{-1}(\pmb\sigma_{t_1})$
is a positive semi-definite matrix for $t_1<t_2<0$, we infer from Corollary \ref{monotonicity equation}  and Lemma  \ref{continuity criteria}  that
 $\lambda_n(\pmb\sigma_{t_1})\geq \lambda_n(\pmb\sigma_{t_2})$  for each $1\leq n \leq Nd-r^0$. Hence, by   Lemma \ref{monotonicity-to-dis} (1),
$\lim_{t\rightarrow {0}^-}\lambda_{n}(\pmb\sigma_{t})=-\infty, 1\leq n\leq r^--r_1^-.$
 Then we get by Lemma \ref{integer} that
 $\tilde {\lambda}_{n}(\mathcal{V}_\varepsilon^{(r^0,r^+,r^-)})={\lambda}_{n}(\mathcal{V}_\varepsilon^{(r^0,r^+,r^-)})\subset(-\infty,c_1)$, $1\leq n\leq r^--r_1^-$,
satisfy (\ref{er-bigger1}).
This, along with Lemma \ref{infinity-finity}, yields (\ref{er-bigger2}).

{\bf Step 2.} Show that (\ref{eboth bigger1})--(\ref{eboth bigger3}) hold
for $r^\pm>r_1^\pm$.

By (\ref{er-bigger1}) and (\ref{er+bigger2}), we have
 $\lambda_n(\tilde{\pmb\sigma}_1) \in(-\infty,c_1)$ with $1\leq n\leq r^--r_1^-$,   and $\lambda_{n}(\tilde{\pmb\sigma}_2)\in(c_2,+\infty)$ with $Nd-r_1^0< n\leq Nd-r_1^0+(r^+-r_1^+)$ for any fixed
$\tilde{\pmb\sigma}_1\in \mathcal{V}_\varepsilon^{(r_1^0-(r^--r_1^-),r_1^+,r^-)}$ and $\tilde{\pmb\sigma}_2\in \mathcal{V}_\varepsilon^{(r_1^0-(r^+-r_1^+),r^+,r_1^-)}$.
Then we infer from (\ref{er-bigger2})--(\ref{er+bigger1}) that
\begin{align*}
\lim\limits_{\mathcal{E}^{(r^0,r^+,r^-)}\ni\pmb{\sigma}\rightarrow\tilde{\pmb{\sigma}}_1}\lambda_{n}(\pmb{\sigma})
=&\lambda_n(\tilde{\pmb{\sigma}}_1),\;1\leq n\leq Nd-r^0-(r^+-r_1^+),\\
\lim\limits_{\mathcal{E}^{(r^0,r^+,r^-)}\ni\pmb{\sigma}\rightarrow\tilde{\pmb{\sigma}}_2}\lambda_{n}(\pmb{\sigma})
=&\lambda_{n-(r^--r_1^-)}(\tilde{\pmb{\sigma}}_2),\;r^--r_1^-< n\leq Nd-r^0.
\end{align*}
Since $Nd-r^0-(r^+-r_1^+)\geq r^--r_1^-$, we obtain that there exists
$\tilde{\pmb{\sigma}}_3\in \mathcal{V}_\varepsilon^{(r^0,r^+,r^-)}$ such that   $\lambda_n(\tilde{\pmb{\sigma}}_3) \in(-\infty,c_1)$ with $1\leq n\leq r^--r_1^-$, and  there exists
$\tilde{\pmb{\sigma}}_4\in \mathcal{V}_\varepsilon^{(r^0,r^+,r^-)}$ such that
 $\lambda_n(\tilde{\pmb{\sigma}}_4) \in(c_2,+\infty)$ with $Nd-r^0-(r^+-r_1^+)<n\leq Nd-r^0$.
This implies that
\begin{align*}
\sharp_1(\sigma({\pmb{\sigma}},{\mathbf{A}})\cap(-\infty,c_1))=r^--r_1^-,\;\;
\sharp_1(\sigma({\pmb{\sigma}},{\mathbf{A}})\cap(c_2,+\infty))=r^+-r_1^+
\end{align*}
for all $\pmb{\sigma}\in \mathcal{V}_\varepsilon^{(r^0,r^+,r^-)}$.
 Then Lemma \ref{integer} (2) ensures that    (\ref{eboth bigger1}) and (\ref{eboth bigger3}) hold. Finally, (\ref{eboth bigger2}) is obtained by  Lemma \ref{infinity-finity}.
\end{proof}

\section{Applications to $d$-dimensional Sturm-Liouville problems of Atkinson type}

Consider the $d$-dimensional Sturm-Liouville problem of Atkinson type with $d\geq1$. The continuous Sturm-Liouville equation is
\begin{align}\label{continuous Sturm-Liouville equation}
-(\hat Py')'+\hat Q y=\lambda \hat Wy \;\textrm{on}\;(a,b),
\end{align}
where  $\hat P, \hat Q$ and $\hat W$ are $d\times d$ Hermitian matrix-valued   functions on $[a,b]$,
and
\begin{align*}
\hat P^{-1},\hat Q, \hat W\in L((a,b),\mathbb{C}^{d\times d}).
\end{align*}
The self-adjoint boundary condition is given by
\begin{equation}\label{boundary condition continuous}
  A\left(
     \begin{array}{c}
       -y(a) \\
       y(b) \\
     \end{array}
   \right)
   +B\left(
     \begin{array}{c}
       (\hat Py')(a) \\
     ( \hat P y')(b) \\
     \end{array}
   \right)=0,
\end{equation}
where $A$  and  $B$  are $2d\times 2d$ complex  matrices, where $A$ and $B$ satisfy (\ref{4}).

(\ref{continuous Sturm-Liouville equation}) is said to be of Atkinson type if
there exists a partition of the interval $(a,b)$,
$$a=a_0<b_0<a_1<b_1<\cdots<a_N<b_N=b$$
for some $N>1$ such that
\begin{align}\label{atkinson condition1}
\hat P^{-1}\equiv0 \;\textrm{on}\;[a_{i},b_{i}],\;\hat W_{i}:=\int_{a_{i}}^{b_{i}}\hat W(s) ds\in\mathcal{P}_d(\mathbb{C}),\; 0\leq i\leq N,
\end{align}
and
\begin{align}\label{atkinson condition2}
\hat Q\equiv\hat W\equiv0 \;\textrm{on}\;[b_{j-1},a_{j}],\;\hat P^{-1}_{j}:=\int_{b_{j-1}}^{a_{j}}\hat P^{-1}(s) ds\; \textrm{is invertible},\; 1\leq j\leq N.
\end{align}
A $d$-dimensional  Sturm-Liouville problem is said to be  of Atkinson type if it consists of (\ref{continuous Sturm-Liouville equation}) of Atkinson type and a self-adjoint boundary condition. $1$-dimensional case has been studied in \cite{Ao-Sun-Zettl2015,Atkinson1,Kong-Volkmer-Zettl,Kong4}. In this section, we always assume that (\ref{continuous Sturm-Liouville equation})--(\ref{boundary condition continuous}) is of Atkinson type.
The space of  Sturm-Liouville equations of Atkinson type is
\begin{equation*}
 \hat \Omega := \{ \left(\hat P^{-1}, \hat Q,\hat W\right)\in  (L((a,b),\mathbb{C}^{d\times d}))^{3}: \, (\ref{atkinson condition1})\text{--}(\ref{atkinson condition2}) \;{\rm hold}\}
\end{equation*}
with topology induced by $(L((a,b),\mathbb{C}^{d\times d}))^{3}$.
$\pmb{\hat\omega}=\left(\hat P^{-1},\hat  Q, \hat W\right)$ is used for an element in $\hat \Omega$.
Note that the space of self-adjoint boundary conditions is also  $\mathcal{ B}^{\C}$ defined by  (\ref{the space of self-adjoint boundary conditions}).
Set
$$\hat Q_i:=\int_{a_i}^{b_i}\hat Q(s) ds,\;\;\;\;0\leq i\leq N.$$
Let $u=y$ and $v=\hat Py'$. Then (\ref{continuous Sturm-Liouville equation}) is transformed to
\begin{align}\label{one order}
\left\{
\begin{array}
[c]{llll}%
u'&=&\hat P^{-1}v,\\
v'&=&(\hat Q-\lambda \hat W)u,
\end{array}
\right.
\end{align}
on $(a,b)$.
It follows from (\ref{atkinson condition1})--(\ref{atkinson condition2}) that if $(u,v)$ is a solution of (\ref{one order}), then $u(t)\equiv u_i\in\mathcal{M}_{d\times1}$ is a constant vector on $[a_i,b_i]$, $0\leq i\leq N$, and $v(t)\equiv v_j\in\mathcal{M}_{d\times1}$ is a constant vector on $[b_{j-1},a_j]$, $1\leq j\leq N$.
Furthermore, we define
\begin{align}\label{def-u-1-uN}
v_{0}=v(a),\;v_{N+1}=v(b),\;u_{-1}=u_0-v_0,\; u_{N+1}=u_N+v_{N+1}.
\end{align}

We construct a $d$-dimensional discrete Sturm-Liouville problem as follows:
\begin{equation}\label{Sturm-Liouville equation constructed}
  -\nabla (\hat P_{i+1}\Delta u_i)+\hat Q_iu_i=\ld\hat W_iu_i,\quad 0\leq i\leq N,
\end{equation}
where $\hat P_{N+1}=\hat P_0=I_{d}$, and a  boundary condition
\begin{equation}\label{boundary condition constructed}
  A\left(
     \begin{array}{c}
       -u_0 \\
       u_N \\
     \end{array}
   \right)
   +B\left(
     \begin{array}{c}
       \Delta u_{-1} \\
       \Delta u_N \\
     \end{array}
   \right)=0,
\end{equation}
where $A$ and $B$ are given in  (\ref{boundary condition continuous}).
By writing $A$ and $B$ into the form (\ref{A1A2B1B2}), direct computation implies that (\ref{boundary condition constructed}) is equivalent to the standard discrete boundary condition:
\begin{equation}\label{boundary condition transform}
  (A_1,\;A_2)\left(
     \begin{array}{c}
       -u_{-1} \\
       u_N \\
     \end{array}
   \right)
   +(B_1-A_1,\;B_2)\left(
     \begin{array}{c}
       \Delta u_{-1} \\
       \Delta u_N \\
     \end{array}
   \right)=0.
\end{equation}
Now we show that (\ref{continuous Sturm-Liouville equation})--(\ref{boundary condition continuous}) is equivalent to the constructed discrete Sturm-Liouville problem above.

\begin{lemma}\label{atkinson-transform-discrete equa}
 (1) (\ref{Sturm-Liouville equation constructed})--(\ref{boundary condition constructed}) is a self-adjoint discrete Sturm-Liouville problem.

(2) (\ref{continuous Sturm-Liouville equation})--(\ref{boundary condition continuous}) is equivalent to (\ref{Sturm-Liouville equation constructed})--(\ref{boundary condition constructed}).
\end{lemma}

\begin{proof} Firstly, we show that (1) holds.
Since $A$ and $B$ satisfy (\ref{4}), we have
\begin{align*}
A_1B_1^*+A_2 B_2^*=B_1A_1^*+B_2A_2^*, \;\textrm{and}\;\; \textrm{rank} (A, B)=2d.
\end{align*}
Thus
\begin{align*}
&(A_1,A_2)\begin{pmatrix}B_1^*-A_1^*\\B_2^*\end{pmatrix}=A_1B_1^*-A_1A_1^*+A_2B_2^*\\
=&B_1A_1^*-A_1A_1^*+B_2A_2^*=(B_1-A_1,B_2)\begin{pmatrix}A_1^*
\\A_2^*\end{pmatrix},
\end{align*}
and
\begin{align*}
\textrm{rank} (A_1,A_2,B_1-A_1,B_2)=\textrm{rank} \left((A,B)\begin{pmatrix}I_d&&-I_d&\\&I_d&&\\&&I_d&\\&&&I_d\end{pmatrix}\right)=2d.
\end{align*}
It follows that  (\ref{boundary condition transform}) is a  self-adjoint boundary condition. Since
$\hat P_{j+1},\hat Q_i, \hat W_i$ are Hermitian, $\hat P_{j+1}$ is invertible, and $\hat W_i\in\mathcal{P}_d(\mathbb{C})$ for $0\leq i\leq N$ and $-1\leq j\leq N$, we have
\begin{align*}
\pmb{\tau}:=\left(\{\hat P_j^{-1}\}_{j=0}^{N+1}, \{\hat Q_i\}_{i=0}^{N}, \{\hat  W_i\}_{i=0}^{N}\right)\in\Omega_{N+1}^\C.
\end{align*}
Hence, (1) holds.

Next, we prove (2). It suffices to show that (\ref{one order}) with (\ref{boundary condition continuous})  is equivalent to (\ref{Sturm-Liouville equation constructed})--(\ref{boundary condition constructed}).  Let $(u,v)$ be a solution   of (\ref{one order}).
Since $v\equiv v_i$ is a constant vector on $[b_{i-1},a_{i}]$, we have
\begin{align*}
u_i-u_{i-1}=u(a_i)-u(b_{i-1})=\int_{b_{i-1}}^{a_i}u'(s)ds=\int_{b_{i-1}}^{a_i}\hat P^{-1}(s) v(s)ds=\hat P_i^{-1}v_i
\end{align*}
for any $1\leq i\leq N$, which, together with  (\ref{def-u-1-uN}) and the fact that $\hat P_{N+1}=\hat P_0=I_{d}$, yields that
\begin{align}\label{hat u}
\hat P_i(u_i-u_{i-1})=v_i,\;\;0\leq i\leq N+1.
\end{align}
Since $u\equiv u_j$ is a constant vector on $[a_{j},b_{j}]$, we obtain
\begin{align}\label{hat v}
&v_{j+1}-v_{j}=v(b_j)-v(a_{j})\\\nonumber
=&\int^{b_{j}}_{a_j}v'(s)ds=\int_{a_{j}}^{b_j}(\hat Q(s)-\lambda \hat W(s))u(s)ds=(\hat Q_j-\lambda \hat W_j)u_j
\end{align}
for any $0\leq j\leq N$. Then (\ref{Sturm-Liouville equation constructed}) is obtained by combining (\ref{hat u})--(\ref{hat v}).

Conversely, let $\{u_i\}_{i=-1}^{N+1}$ be a solution of (\ref{Sturm-Liouville equation constructed}) and define
$v_i=\hat P_i(u_i-u_{i-1})$ for $0\leq i\leq N+1$.
Let $u(t)=u_i$ for all $t\in[a_i,b_i]$ and $0\leq i\leq N$, $v(t)=v_j$ for all $t\in[b_{j-1},a_j]$ and $1\leq j\leq N$,  $v(a)=v_0$, and
\begin{align*}
u(t)&=u(b_{j-1})+\int_{b_{j-1}}^t\hat P^{-1}(s)v_j ds,\; t\in[b_{j-1},a_j],\\
v(t)&=v(a_{i})+\int_{a_i}^{t}(\hat Q(s)-\lambda\hat W(s))u_ids,\;t\in[a_{i},b_i].
\end{align*}
Then $(u,v)$ is a solution of (\ref{one order}).

Moreover, $y(a)=u_0$, $y(b)=u_{N}$, $(\hat P y')(a)=v(a)=u_0-u_{-1}=\Delta u_{-1}$ and $(\hat P y')(b)=v(b)=u_{N+1}-u_{N}=\Delta u_N$. Thus
(\ref{boundary condition continuous}) is equivalent to (\ref{boundary condition constructed}).
\end{proof}

 By transforming the  Sturm-Liouville problem of Atkinson type to the discrete case, we can now determine  the number of eigenvalues in the following lemma, which generalizes Theorems 2.1 and 3.1 in \cite{Kong4} for  1-dimension to any dimension.

\begin{lemma}\label{atkinson-number of eigenvalue}
 Let $(\pmb{\hat\omega},\mathbf{A})\in\hat \Omega\times\mathcal{ B}^{\C}$ with $\mathbf{A}$ given in (\ref{A1A2B1B2}). Then
the eigenvalues of $(\pmb{\hat\omega},\mathbf{A})$, including multiplicities, are the same as those of $(\pmb\tau,\mathbf{A})=(\pmb\tau,\mathbf{C})$, and
\begin{align*}
\sharp_1(\sigma(\pmb{\hat\omega},\mathbf{A}))=\sharp_1(\sigma(\pmb\tau,\mathbf{A}))=\sharp_1(\sigma(\pmb\tau,\mathbf{C}))=(N-1)d+{\rm rank}\; (B),
\end{align*}
where $\pmb\tau\in\Omega_{N+1}^\mathbb{C}$ is the transformed discrete Sturm-Liouville equation by $\pmb{\hat\omega}$, and
\begin{align*}
\mathbf{C}=[(A_1,A_2)|(B_1-A_1,B_2)].
 \end{align*}
\end{lemma}

\begin{remark}\label{rem-AC}
Note that here $\mathbf{A}$ is under the basis  $(-u_0,u_N, \Delta u_{-1},\Delta u_N)$ when we write  $(\pmb\tau,\mathbf{A})$, while
$\mathbf{C}$ is under the standard basis  $(-u_{-1},u_N, \Delta u_{-1},\Delta u_N)$ when we write $(\pmb\tau,\mathbf{C})$. In this sense, $(\pmb\tau,\mathbf{A})=(\pmb\tau,\mathbf{C})$.
\end{remark}
\begin{proof}
By Lemma \ref{atkinson-transform-discrete equa} (2), $(\pmb{\hat\omega},\mathbf{A})$ is equivalent to $(\pmb{\tau},\mathbf{A})=(\pmb{\tau},\mathbf{C})$, and   $\hat P_0=I_d$.
 Then applying Lemma \ref{number of eigenvalues} to $(\pmb{\tau},\mathbf{C})$, we have
\begin{align*}
&\sharp_1(\sigma(\pmb{\hat\omega},\mathbf{A}))= \sharp_1(\sigma(\pmb{\tau},\mathbf{C}))\\
=&((N+1)-2)d+{\rm rank} (A_1\hat P_0^{-1}+(B_1-A_1),B_2)=(N-1)d+{\rm rank} (B_1,B_2).
\end{align*}
This completes the proof.
\end{proof}
Then we study singularity of the $n$-th eigenvalue of  $d$-dimensional  Sturm-Liouville problems of Atkinson type.
  We first claim in Proposition \ref{Atkinson type equation continuous} below that there is no singularity of the $n$-th eigenvalue on the equations.
In fact, for a fixed $\mathbf{A}\in\mathcal{B}^\mathbb{C}$, we infer from Lemma \ref{atkinson-number of eigenvalue} that $\sharp_1(\sigma(\pmb{\hat\omega},\mathbf{A}))$  is independent of
$\pmb{\hat\omega}\in\hat\Omega$. This,
together with Lemma \ref{continuity criteria}, implies the following result.
\begin{proposition}\label{Atkinson type equation continuous}
Fix $\mathbf{A}=[A\;|\;B]\in\mathcal{B}^\mathbb{C}$. Then  the $n$-th eigenvalue is continuous on the whole space of  Sturm-Liouville equations of Atkinson type $ \hat \Omega$ for all $1\leq n\leq \sharp_1(\sigma(\pmb{\hat\omega},\mathbf{A}))=(N-1)d+\textrm{rank}\;(B).$
\end{proposition}
Next, we consider singularity of the $n$-th eigenvalue on the boundary conditions.
Lemma \ref{atkinson-number of eigenvalue} indicates that it suffices to study singularity of the $n$-th eigenvalue of $(\pmb\tau,\mathbf{C})=(\pmb\tau,\mathbf{A})$ for the fixed $\pmb\tau$.
 The coupled term $B_1-A_1$ in the standard boundary condition $\mathbf{C}$
 makes it hard to apply
 Theorem \ref{main result discrete case} to   $(\pmb\tau,\mathbf{C})$ directly.
We shall apply the method developed  in Sections 3--4 to the discrete Sturm-Liouville problem  $(\pmb\tau,\mathbf{A})$ and provide a direct proof here.
In order to study singularity in a certain direction, we  need  the derivative formula of a continuous simple eigenvalue branch.

\begin{lemma}\label{derivative formula2}
Fix $\pmb{\tau}\in \Omega_{N+1}^\mathbb{C}$. Let $\ld_*$ be a simple eigenvalue of $(\pmb\tau, \mathbf{A})$
for  $\mathbf{A}\in \mathcal{O}_{K}^\mathbb{C}$,
$u \in l[-1, N + 1]$ be a normalized eigenfunction for $\ld_*$, and $\Ld$
be the continuous simple eigenvalue branch  through $\ld_*$.
Then we have the following derivative formula
\begin{equation*}
  d \Ld |_{\mathbf{A}}(H)=Z^*E_{K,1}^*H E_{K,1}Z
\end{equation*}
for $H\in\mathcal{H}_{2d}(\mathbb{C})$, where
\begin{equation*}
 Z^T=(-u_0^T,u_N^T,(\Delta u_{-1})^T,(\Delta u_N)^T).
 \end{equation*}
\end{lemma}
\begin{remark}
 Lemma \ref{derivative formula1} is  unable to be directly applied  here due to the different basis.
Note carefully that $Z$ in Lemma \ref{derivative formula2} is different from $Y$ in Lemma \ref{derivative formula1}.
\end{remark}
\begin{proof}
Recall that there exists $S\in\mathcal{H}_{2d}(\mathbb{C})$ such that $\mathbf{A}=[S\;|\;I_{2d}]E_K$. Let $\mathbf{B}=[S+H\;|\;I_{2d}]E_K$ with $H\in\mathcal{H}_{2d}(\mathbb{C})$.
Then by Lemma \ref{continuous choice of eigenfunctions}, there exists an eigenfunction $\tilde u=\{\tilde u_i\}_{i=-1}^{N+1}$ for $\Lambda(\mathbf{B})$ such that $\tilde u\to u$  in $\mathbb{C}^{(N+2)d}$ as $\mathbf{B}\to \mathbf{A}$.
Note that  $\tilde u$ and ${u}$ satisfy
\begin{align*}
-\nabla(\hat P_{i+1}\Delta\tilde{u}_i)+\hat Q_i\tilde{u}_i=\Lambda(\mathbf{B})\hat W_i\tilde{u}_i,\;\;
-\nabla(\hat P_{i+1} \Delta u_i)+\hat Q_i u_i=\Lambda(\mathbf{A})\hat W_i u_i,\;0\leq i\leq N,
\end{align*}
and thus
\begin{align*}
(\Lambda(\mathbf{B})-\Lambda(\mathbf{A}))\sum_{i=0}^N\tilde u_i^* \hat W_i u_i
=&\lfloor u_{-1},\tilde u_{-1}\rfloor-\lfloor u_{N},\tilde u_{N}\rfloor\\
=&\tilde u_0^*u_{-1}-\tilde u_{-1}^*u_0-\lfloor u_{N},\tilde u_{N}\rfloor\\
=&(\Delta \tilde u_{-1})^*u_0-\tilde u_0^*(\Delta u_{-1})-\lfloor u_{N},\tilde u_{N}\rfloor,
\end{align*}
where $\lfloor u_i,\tilde u_i\rfloor=(\Delta\tilde u_i)^*\hat P_{i+1}u_i-\tilde u_i^*\hat P_{i+1}\Delta u_i.$
$\mathbf{A}$ and $\mathbf{B}$ tell us that $SE_{K,1}Z+E_{K,2}Z=0$  and $(S+H)E_{K,1}\tilde Z+E_{K,2}\tilde Z=0$.
Then we infer from (\ref{EKrelation}) that
\begin{align*}
(\Lambda(\mathbf{B})-\Lambda(\mathbf{A}))\sum_{i=0}^N\tilde u_i^*\hat W_iu_i
=&\tilde Z^*E_K^*J_{2d}^*E_KZ=\tilde Z^*E_{K,1}^*H E_{K,1}Z.
\end{align*}
This completes the proof.
\end{proof}

As a consequence, we get the following conclusion.
\begin{corol}\label{e-monotonicity}
 Let $\Lambda$ be a continuous eigenvalue branch defined on $\mathcal{U}\subset \mathcal{O}_K^{\C}$.
Then $\Lambda(\mathbf{A})\leq\Lambda(\mathbf{B})$ if $\mathbf{A},\mathbf{B}\in\mathcal{U}$ and $S(\mathbf{B})-S(\mathbf{A})$ is positive semi-definite.
\end{corol}

Let the $k$-th layer in $\mathcal{B}^\mathbb{C}$ be defined as
\begin{align*}
\hat \Sigma_{k}:=&\{\mathbf{A}\in\mathcal{B}^\mathbb{C}\mid r^0( B)=k\}, \; 0\leq k\leq 2d.
\end{align*}
Then the following result is a direct consequence of Lemma \ref{atkinson-number of eigenvalue}.
\begin{corol}\label{atlinson number of eigenvalue b}
Fix  $\pmb{\hat\omega}\in \hat \Omega$. Then

 (1) $\sharp_1(\sigma(\pmb{\hat\omega},\mathbf{A}))=(N+1)d-k$ for all $\mathbf{A}\in\hat \Sigma_{k}$.

 (2) $\sharp_1(\sigma(\pmb{\hat\omega},\mathbf{A}))=(N+1)d$ for all $\mathbf{A}\in\mathcal{O}_\emptyset^{\mathbb{C}}$.
\end{corol}
 For a nonempty subset $K=\{n_1,\cdots,n_{m_0}\}\subset \{1,\cdots, 2d\}$, we  define  $\hat E_0=(\hat e_{n_1},\cdots,\hat e_{n_{m_0}})$,
 and \begin{align}\label{def-SK}
 S^A_K(\mathbf{A})=\hat E_0^*S(\mathbf{A})\hat E_0,\end{align}
where $\mathbf{A}\in\mathcal{O}_K^\mathbb{C}$, and $\hat e_i$ is the $i$-th column of $I_{2d}.$
The divided area is defined by
\begin{align*}
J^{(r^0,r^+,r^-)}_{{K}}
:=\{\mathbf{A}\in\mathcal{O}_{K}^{\mathbb{C}}|r^0( S^A_K(\mathbf{A}))=r^0,r^\pm( S^A_K(\mathbf{A}))=r^\pm\}
\end{align*}
for three nonnegative integers $r^0,$ $r^+$ and $r^-$ satisfying $r^0+r^++r^-=m_0$.  Then we are ready to provide the complete characterization of singularity of the $n$-th eigenvalue for the Atkinson type.

\begin{theorem}\label{main result for Atkinson type}
 Fix $\pmb{\hat\omega}\in\hat\Omega$.

 (1) Let $0\leq k\leq 2d$. Then the restriction of $\lambda_n$ to $\hat\Sigma_k$ is continuous for any $1\leq n\leq (N+1)d-k$.
Moreover,  the restriction of $\lambda_n$ to $\mathcal{O}_\emptyset^{\mathbb{C}}$ is continuous for any $1\leq n\leq (N+1)d$.

(2) Consider the restriction of $\lambda_n$ to $\mathcal{O}_K^{\mathbb{C}}$, where $\emptyset\neq K\subset\{1,\cdots,2d\}$.
 Let  $0\leq r^0<r^0_1\leq \sharp(K)$ and $r^\pm\geq r_1^\pm$. Then
for any $\mathbf{A}\in J^{(r_1^0,r_1^+,r_1^-)}_{{K}}$, we have
\begin{align}\label{atkinson both bigger1}
\lim\limits_{J^{(r^0,r^+,r^-)}_{{K}}\ni\mathbf{B}\to\mathbf{A}}\lambda_n(\mathbf{B})&=-\infty,\;1\leq n\leq r^+-r_1^+,\\
\label{atkinson both bigger2}    \lim_{J_K^{(r^0,r^+,r^-)}\ni\mathbf{B}\to\mathbf{A}}\ld_{n}(\mathbf{B})
    &=\ld_{n-(r^+-r^+_1)}(\mathbf{A}),\;  r^+-r_1^+< n\le (N+1)d-r^0-( r^--r_1^-),\\
\label{atkinson both bigger3}      \lim_{J_K^{(r^0,r^+,r^-)}\ni\mathbf{B}\to \mathbf{A}}\ld_{n}(\mathbf{B})&=+\infty, \; (N+1)d-r^0-( r^--r_1^-)< n\le (N+1)d-r^0.
\end{align}
Consequently, the singular set is $\cup_{1\leq k\leq 2d}\hat\Sigma_k$.
 \end{theorem}
\begin{remark}
Note that $S_K^A(\mathbf{A})$ is independent of the  Sturm-Liouville equations of Atkinson type, while $S_K^D(\mathbf{A})$ defined in (\ref{Hermite})--(\ref{Hermite2}) is indeed involved heavily with the coefficient $P_0^{-1}$ of the discrete equations.
\end{remark}

\begin{proof}
We study the equivalent discrete Sturm-Liouville problem $(\pmb{\tau},\mathbf{A})$.
(1) is straightforward by  Lemma \ref{continuity criteria} and Corollary \ref{atlinson number of eigenvalue b}.
Next, we show that  (2) holds.
Choose  $c_1,c_2\in\mathbb{R}$  such that $\sharp_1(\sigma({\pmb{\tau}},\mathbf{A})\cap(c_1,c_2))=(N+1)d-r_1^0$.
 It follows from Lemma \ref{small perturbation} that
 $\sharp_1(\sigma({\pmb{\tau}},\mathbf{B})\cap(c_1,c_2))=(N+1)d-r_1^0$
with $c_1,c_2\notin\sigma({\pmb{\tau}},\mathbf{B})$ for all $\mathbf{B}\in\mathcal{U}_\varepsilon$, defined in Lemma \ref{neighborhood}, and $\varepsilon>0$ small enough.
Lemma 7.2 in \cite{Hu-Liu-Wu-Zhu2018} implies that  $\hat{\mathcal{U}}_\varepsilon^{(r^0,r^+,r^-)}=\mathcal{U}_\varepsilon\cap J_K^{(r^0,r^+,r^-)}$ is path connected.
Note that
\begin{align}\label{number out of c1c2}
\sharp_1(\sigma({\pmb{\tau}},\mathbf{B})\cap((-\infty,c_1)\cup(c_2,+\infty)))=r_1^0-r^0
\end{align}
  for  $\mathbf{B}\in\hat{\mathcal{U}}_\varepsilon^{(r^0,r^+,r^-)}$.


 We show that if  $r^+>r_1^+$ and $ r^-=r_1^-$, then (\ref{r+bigger1}$'$)--(\ref{r+bigger2}$'$) hold. Similarly, if $ r^+=r_1^+$ and $ r^->r_1^-$,
then  (\ref{r-bigger1}$'$)--(\ref{r-bigger2}$'$) hold. Here (\ref{r+bigger1}$'$)--(\ref{r-bigger2}$'$) are defined as  (\ref{r+bigger1})--(\ref{r-bigger2}) with
$\mathcal{B}_K^{(r^0,r^+,r^-)}$ and $N$  replaced by  ${J}_K^{(r^0,r^+,r^-)}$ and $N+1$.
Let  $\hat M\in\mathcal{M}_{m_0}$ be the unitary matrix such that
$ S^A_K(\mathbf{A})=\hat M{\rm diag} \{\tilde\nu_1,\cdots,\tilde \nu_{m_0}\}\hat M^*$,
where $m_0=\sharp(K)$ and $\tilde\nu_1=\cdots=\tilde\nu_{r_1^0}=0$. Define $\hat{\mathbf{B}}_t=[S(\hat{\mathbf{B}}_t)\,|\,I_{2d}]E_K$ with
$$S^A_K(\hat{\mathbf{B}}_t)= S^A_K(\mathbf{A})+\hat M\begin{pmatrix}tI_{r^+-r_1^+}&\\&0_{m_0-({r^+-r_1^+})}
\end{pmatrix}\hat M^*,$$
$t\geq0$ is sufficiently small, and $s_{ij}(\hat{\mathbf{B}}_t)=s_{ij}(\mathbf{A})$ if $i\in\{1,\cdots,2d\}\setminus K$ or $j\in\{1,\cdots,2d\}\setminus K$.
  Since
 $S(\hat{\mathbf{B}}_{t_2})-S(\hat{\mathbf{B}}_{t_1})$
is a positive semi-definite matrix for $t_2>t_1>0$, it follows from Lemma \ref{continuity criteria} and Corollary \ref{e-monotonicity} that
 $\lambda_n(\hat{\mathbf{B}}_{\cdot})$ is non-decreasing on $t\in(0,\varepsilon)$ for each $1\leq n \leq (N+1)d-r^0$, where $\varepsilon>0$ is small enough. Hence, by   Lemma \ref{monotonicity-to-dis} (4),
$\lim_{t\rightarrow {0}^+}\lambda_{n}(\hat{\mathbf{B}}_t)=-\infty, 1\leq n\leq r^+-r_1^+=r_1^0-r^0.$ This,
 along with Lemma \ref{integer}  and (\ref{number out of c1c2}), yields that
 ${\lambda}_{n}(\hat{\mathcal{U}}_\varepsilon^{(r^0,r^+,r^-)})\subset(-\infty,c_1)$, $1\leq n\leq r^+-r_1^+$,
and (\ref{r+bigger1}$'$) holds.
Then we get by Lemma \ref{infinity-finity} that (\ref{r+bigger2}$'$) holds.

Finally, we prove
(\ref{atkinson both bigger1})--(\ref{atkinson both bigger3})
for $r^\pm>r_1^\pm$.
It follows from (\ref{r+bigger1}$'$) and (\ref{r-bigger2}$'$) that
 $\lambda_n(\tilde{\mathbf{A}}_1) \in(-\infty,c_1)$ with $1\leq n\leq r^+-r_1^+$,   and $\lambda_{n}(\tilde{\mathbf{A}}_2)\in(c_2,+\infty)$ with $(N+1)d-r_1^0< n\leq (N+1)d-r_1^0+(r^--r_1^-)$ for any fixed
$\tilde{\mathbf{A}}_1\in \hat{\mathcal{U}}_\varepsilon^{(r_1^0-(r^+-r_1^+),r^+,r_1^-)}$ and $\tilde{\mathbf{A}}_2\in \hat{\mathcal{U}}_\varepsilon^{(r_1^0-(r^--r_1^-),r_1^+,r^-)}$.
Furthermore, we have by (\ref{r-bigger1}$'$)
that there exists
$\tilde{\mathbf{A}}_3\in \hat{\mathcal{U}}_\varepsilon^{(r^0,r^+,r^-)}$ such that   $\lambda_n(\tilde{\mathbf{A}}_3) \in(-\infty,c_1)$ with $1\leq n\leq r^+-r_1^+$.
It follows from (\ref{r+bigger2}$'$) that there exists
$\tilde{\mathbf{A}}_4\in\hat {\mathcal{U}}_\varepsilon^{(r^0,r^+,r^-)}$ such that
 $\lambda_n(\tilde{\mathbf{A}}_4) \in(c_2,+\infty)$ with $r^+-r_1^+\leq (N+1)d-r^0-(r^--r_1^-) <n\leq (N+1)d-r^0$.
Then we get by Lemma \ref{integer} (1) that $\lambda_n(\hat {\mathcal{U}}_\varepsilon^{(r^0,r^+,r^-)})\subset(-\infty,c_1)$ for $1\leq n\leq r^+-r_1^+$, and
$\lambda_{n}(\hat{\mathcal{U}}_\varepsilon^{(r^0,r^+,r^-)})\subset (c_2,+\infty) $ for  $(N+1)d-r^0-(r^--r_1^-)<n\leq (N+1)d-r^0$.
Thanks to Lemma \ref{integer} (2), we get  (\ref{atkinson both bigger1}) and (\ref{atkinson both bigger3}).  Then
 (\ref{atkinson both bigger2}) is a direct consequence of Lemma \ref{infinity-finity}.
The proof is complete.
\end{proof}

\section{Conclusions  and comparisons of singularity of the $n$-th eigenvalue  among   continuous case, discrete case,  and Atkinson type}

In this section, we compare singularity of the $n$-th eigenvalue among  the  Sturm-Liouville problems for the continuous case (1.1)--(1.2) in \cite{Hu-Liu-Wu-Zhu2018}, the discrete case (\ref{Sturm-Liouville equation})--(\ref{boundary condition }), and  the Atkinson type (\ref{continuous Sturm-Liouville equation})--(\ref{boundary condition continuous}).

(i). Comparison of singularity on boundary conditions.

 According to Theorem 7.1 in \cite{Hu-Liu-Wu-Zhu2018}, Theorem \ref{main result discrete case}, and Theorem \ref{main result for Atkinson type},
 the singularity on  the  boundary conditions is determined by the constructed Hermitian matrices, which are  $S_K^C(\mathbf{A})$ given in (4.2) of \cite{Hu-Liu-Wu-Zhu2018} for the continuous case,  $S_K^D(\mathbf{A})$ defined in (\ref{Hermite})--(\ref{Hermite2}) for the discrete case,  and $S_K^A(\mathbf{A})$
defined in (\ref{def-SK}) for the Atkinson type, where $\mathbf{A}\in\mathcal{O}_K^\mathbb{C}$.

For the continuous case, Theorem 7.1 in \cite{Hu-Liu-Wu-Zhu2018} tells us that  the first $m_c$ eigenvalues jump to  $-\infty$ as a path of boundary conditions from the lower layer of $\mathcal{O}_K^\mathbb{C}$ tends to a given boundary condition in the upper layer. Here the jump number $m_c$ is exactly the number of transitional eigenvalues (from positive to zero) of the determined Hermitian matrices. It is further shown that this number is the Maslov index of the path of boundary conditions in a forthcoming paper.

For the discrete case, Theorem \ref{main result discrete case} indicates that not only the first $m_d^-$ eigenvalues jump to $-\infty$, but the last  $m_d^+$ eigenvalues also blow up to $+\infty$  as a path of boundary conditions from the lower layer  tends to a given boundary condition in the upper layer. Here the jump number $m_d^-$ has the similar meaning as $m_c$ in the continuous case, while $m_d^+$ is
the number of transitional eigenvalues (from negative to zero) of the determined Hermitian matrices.

For the Atkinson type, Theorem \ref{main result for Atkinson type} renders  both similar jump phenomena to $\pm\infty$ with numbers $m_a^{\pm}$ as in the discrete case.
However, $m_d^{\pm}\neq m_a^{\pm}$ in general, which is due to the fact that  the determined Hermitian matrices are different, i.e., $S_K^D(\mathbf{A})\neq S_K^A(\mathbf{A})$.
It is also interesting to see that the determined  Hermitian matrices for the continuous case and the Atkinson type are the same, i.e., $S_K^C(\mathbf{A})=S_K^A(\mathbf{A})$. The singular set  in the Atkinson type coincides with  that in the continuous case. This, in particular, provides a direct consequence: $m_c=m_a^-$.

The determined Hermitian matrix is independent of  coefficients of the Sturm-Liouville equations for the continuous case and the Atkinson type, while the coefficient $P_0^{-1}$ involves heavily in the Hermitian matrix for the discrete case.
In addition, the order of the determined Hermitian matrix is  $d+\sharp (K_2)$ for the discrete case, while it is $\sharp( K)$ for the continuous case and the Atkinson type.
This implies that the maximal jump number in the discrete case is always no less than that in the continuous case and the Atkinson type.

(ii). Comparison of singularity on the equations.

Based on Theorem 6.1 in \cite{Hu-Liu-Wu-Zhu2018},  Theorem \ref{main result equation discrete case}, and Proposition \ref{Atkinson type equation continuous}, the $n$-th eigenvalue has no singularity on   coefficients of the Sturm-Liouville equations for the  continuous case and the Atkinson type, while indeed exhibits jump phenomena when  coefficients of the Sturm-Liouville equations vary for the discrete case.

For the discrete case, Theorem \ref{main result equation discrete case} also provides  jump phenomena to both $\pm\infty$ with jump numbers  $\tilde m_d^\pm$ as a path of equations from the lower layer of $\Omega_N^\mathbb{C}$ tends to a given equation in the upper layer. The determined Hermitian matrix is given by $T(\pmb\omega)$ defined in (\ref{Tpmbomega-def1})--(\ref{Tpmbomega-def2}). $\tilde m_d^-$ is
the number of transitional eigenvalues (from negative to zero) of the determined Hermitian matrices, while $\tilde m_d^+$ is
the number of transitional eigenvalues (from positive to zero) of the determined Hermitian matrices. Here the reverse direction for the  transitional eigenvalues in the definitions of $m_d^\pm$ and $\tilde m_d^\pm$ is essentially due to the  opposite  monotonicity of the continuous eigenvalue branches, see Corollaries \ref{monotonicity equation} and \ref{cor-monotonicity-boundary}.

(iii). Comparison of the method in the proof of singularity.

Compared with the continuous case in   \cite{Hu-Liu-Wu-Zhu2018} and \cite{Kong2},
  the singular set in the discrete case  is involved heavily  with  coefficients of the Sturm-Liouville equations. Moreover,
  the finiteness of  spectrum  for the  discrete case or the Atkinson type  makes the method for the continuous case (e.g.,  continuity principle in \cite{Hu-Liu-Wu-Zhu2018,Kong2}) invalid here.
 Compared with the
    $1$-dimensional discrete case in \cite{Zhu2},
the first difficulty is how to divide  areas in each layer of the considered space such that the $n$-th eigenvalue has the same  singularity in a given area. We study  singularity by partitioning and analyzing the local coordinate systems, and provide a Hermitian matrix which can determine the areas' division.
As mentioned in Introduction, our approach to proving the asymptotic behavior of the $n$-th eigenvalue here should be taken as
a generalization of
the method developed for $1$-dimensional discrete  case in
\cite{Zhu2} to any dimension.

Finally, we  list several determined Hermitian matrices as follows in $2$-dimension to exhibit how the difference is  between  the continuous case (Atkinson type) and the discrete case. The orders of the determined Hermitian matrices for  the discrete case are larger than those for the continuous case (Atkinson type) in (2)--(3), (5)--(6) and (8). On the other hand, these orders  are the same in (4), (7) and (9).   Even though, for example, the maximal jump number is $1$ in the continuous case and the Atkinson type, while it is $3$ in the discrete case when $K=\{3\}$. However,  it is both $2$ in any case when $K=\{1,2\}$.

Let
\begin{equation*}
 P_0^{-1}=\left(
     \begin{array}{cc}
       p_1 & p_2 \\
       \bar{p}_2 & p_3 \\
     \end{array}
   \right)
\end{equation*}
for the discrete case.

(1) $K=\emptyset$.
\begin{align*}
S^D_K(\mathbf{A})=\left(
     \begin{array}{cc}
s_{11}+\frac{p_3}{p_{1}p_3-|p_2|^2} & s_{12}-\frac{p_2}{p_{1}p_3-|p_2|^2}  \\
      \bar{s}_{12}-\frac{\bar{p}_2}{p_{1}p_3-|p_2|^2}&  s_{22}+\frac{p_1}{p_{1}p_3-|p_2|^2} \\
     \end{array}
   \right)
\end{align*}
and there are no $S^C_K(\mathbf{A})$ and $ S^A_K(\mathbf{A})$, since there is no singularity for the continuous case and the Atkinson type when $K=\emptyset$.

(2) $K=\{1\}$.
\begin{align*}
S^C_K(\mathbf{A})=S^A_K(\mathbf{A})=(s_{11}),\;\;S^D_K(\mathbf{A})=\begin{pmatrix}
     s_{11}-\frac{p_1p_3-|p_2|^2}{p_3} & s_{12}-\frac{p_2}{p_3}  \\
     \bar{ s}_{12}-\frac{\bar{p}_2}{p_3} & s_{22}+\frac{1}{p_3} \\
     \end{pmatrix}.
\end{align*}

(3) $K=\{3\}$.
\begin{align*}
S^C_K(\mathbf{A})=S^A_K(\mathbf{A})=(s_{33}),\;\;
S^D_K(\mathbf{A})=\left(
     \begin{array}{ccc}
s_{11}+\frac{p_3}{p_{1}p_3-|p_2|^2} & s_{12}-\frac{p_2}{p_{1}p_3-|p_2|^2} & s_{13} \\
      \bar{s}_{12}-\frac{\bar{p}_2}{p_{1}p_3-|p_2|^2}&  s_{22}+\frac{p_1}{p_{1}p_3-|p_2|^2}& s_{23} \\
       \bar{s}_{13}& \bar{s}_{23}& s_{33}\\
     \end{array}
   \right).
\end{align*}

(4) $K=\{1,2\}$.
\begin{align*}
S^C_K(\mathbf{A})=S^A_K(\mathbf{A})=(s_{ij})_{1\leq i,j\leq 2},\;\;
S^D_K(\mathbf{A})=\left(\begin{array}{cc}
s_{11}-p_1 & s_{12}-p_2  \\
      \bar{s}_{12}-\bar{p}_2&  s_{22}-p_3 \\
     \end{array}
   \right).
\end{align*}

(5) $K=\{1,3\}$.
\begin{align*}
S^C_K(\mathbf{A})=S^A_K(\mathbf{A})=(s_{ij})_{i,j\in\{1,3\}},\;\;
S^D_K(\mathbf{A})=\left(
     \begin{array}{ccc}
       s_{11}-\frac{p_1p_3-|p_2|^2}{p_3} & s_{12}-\frac{p_2}{p_3} &s_{13} \\
      \bar{s}_{12}-\frac{\bar{p}_2}{p_3} & s_{22}+\frac{1}{p_3} &s_{23} \\
      \bar{s}_{13} & \bar{s}_{23} &s_{33} \\
     \end{array}
   \right).
\end{align*}

(6) $K=\{3,4\}$.
\begin{align*}
S^C_K(\mathbf{A})=S^A_K(\mathbf{A})=(s_{ij})_{3\leq i,j\leq 4},
S^D_K(\mathbf{A})=\left(
     \begin{array}{cccc}
s_{11}+\frac{p_3}{p_{1}p_3-|p_2|^2} & s_{12}-\frac{p_2}{p_{1}p_3-|p_2|^2} & s_{13} & s_{14}\\
      \bar{s}_{12}-\frac{\bar{p}_2}{p_{1}p_3-|p_2|^2}&  s_{22}+\frac{p_1}{p_{1}p_3-|p_2|^2}& s_{23} & s_{24}\\
       \bar{s}_{13}& \bar{s}_{23}& s_{33}& s_{34}\\
       \bar{s}_{14} & \bar{s}_{24} &\bar{s}_{34}&s_{44}\\
     \end{array}
   \right).
\end{align*}

(7) $K=\{1,2,3\}$.
\begin{align*}
S^C_K(\mathbf{A})=S^A_K(\mathbf{A})=(s_{ij})_{1\leq i,j\leq 3},\;\;
S^D_K(\mathbf{A})=\left(\begin{array}{ccc}
s_{11}-p_1 & s_{12}-p_2  &s_{13} \\
      \bar{s}_{12}-\bar{p}_2&  s_{22}-p_3  &s_{23}\\
      \bar{ s}_{13} &\bar{s}_{23} &s_{33}\\
     \end{array}
   \right).
\end{align*}

(8) $K=\{1,3,4\}$.
\begin{align*}
S^C_K(\mathbf{A})=S^A_K(\mathbf{A})=(s_{ij})_{i,j\in\{1,3,4\}},\;\;
S^D_K(\mathbf{A})=\left(
     \begin{array}{cccc}
      s_{11}-\frac{p_1p_3-|p_2|^2}{p_3} & s_{12}-\frac{p_2}{p_3}&s_{13}&s_{14}\\
     \bar{s}_{12}-\frac{\bar{p}_2}{p_3} & s_{22}+\frac{1}{p_3} &s_{23}&s_{24}\\
      \bar{s}_{13} & \bar{s}_{23} &s_{33}&s_{34}\\
     \bar{s}_{14} & \bar{s}_{24} &\bar{s}_{34}&s_{44}\\
     \end{array}
   \right).
\end{align*}

(9) $K=\{1,2,3,4\}$.
\begin{align*}
S^C_K(\mathbf{A})=S^A_K(\mathbf{A})=(s_{ij})_{1\leq i,j\leq 4},\;\;S^D_K(\mathbf{A})=\left(\begin{array}{cccc}
s_{11}-p_1 & s_{12}-p_2  &s_{13} &s_{14}\\
      \bar{s}_{12}-\bar{p}_2&  s_{22}-p_3  &s_{23}&s_{24}\\
      \bar{ s}_{13} &\bar{s}_{23} &s_{33}&s_{34}\\
       \bar{s}_{14} & \bar{s}_{24} &\bar{s}_{34}&s_{44}\\
     \end{array}
   \right).
\end{align*}

\section*{Acknowledgement}
H. Zhu sincerely thanks Prof. Yiming  Long for his  consistent support and great help.
Part of this
work was done while H. Zhu was visiting the IMS of  CUHK; he
sincerely thanks Prof. Zhouping Xin for his invitation.
 G. Ren is partially supported by  NSFC
(No. 11571202).
 H. Zhu is partially supported  by  NSFC
(No. 11790271),  PITSP (No. BX20180151) and CPSF (No. 2018M630266).

\end{CJK*}


\begin{thebibliography}{99}
\bibitem{Ao-Sun-Zettl2015} J. Ao, J. Sun, A. Zettl,  \textit{Finite spectrum of $2n$th order boundary value problems}, Appl. Math. Lett.  42  (2015), 1--8.
\bibitem{Arnold1967} V. I. Arnold, \textit{Characteristic class entering in quantization conditions}, Funct. Anal.  Appl. 1 (1967),  1--13.
\bibitem{Arnold2000} V. I. Arnold, \textit{The complex Lagrangian Grassmannian}, (Russian)
Funktsional. Anal. i Prilozhen. 34 (2000)  63--65; translation in
Funct. Anal. Appl. 34 (2000), 208--210.
\bibitem{Atkinson1} F. V. Atkinson, \textit{ Discrete and Continuous Boundary Problems}, Academic Press, New York, 1964.

\bibitem{Bohner1} M. Bohner, O. Do\v sl\'y, W. Kratz, \textit{Sturmian and spectral theory for discrete symplectic systems},
Trans. Amer. Math. Soc. 361 (2009), 3109--3123.


\bibitem{Clark1} S. Clark, \textit{A spectral analysis for self-adjoint operators generated by a class of second order difference equations},
J. Math. Anal. Appl. 197 (1996), 267--285.




\bibitem{Everitt1} W. N. Everitt, M. M$\rm{\ddot{o}}$ller, A. Zettl, \textit{ Discontinuous dependence of the
n-th Sturm-Liouville eigenvalues},
in: General Inequalities, Birkh$\rm{\ddot{a}}$user, Basel, 1997.


\bibitem{Hu-Liu-Wu-Zhu2018} X. Hu, L. Liu, L. Wu, H. Zhu,  \textit{ Singularity of the $n$-th eigenvalue of  high dimensional Sturm-Liouville problems}, J. Differential Equations (2018), https://doi.org/10.1016/j.jde.2018.09.028.

\bibitem{Jirari}  A. Jirari, \textit{Second-order Sturm-Liouville difference equations and orthogonal polynomials}, Mem. Amer. Math. Soc.  113  (1995),   x+138.

 \bibitem{Kato1984} T. Kato, \textit{Perturbation Theory for Linear Operators}, 2nd ed., Springer-Verlag, Berlin/Heidelberg/New York/Tokyo, 1984.





\bibitem{Kong-Volkmer-Zettl} Q. Kong, H.  Volkmer,   A. Zettl, \textit{ Matrix representations of Sturm-Liouville problems with finite spectrum}, Results Math.  54  (2009),   103--116.

\bibitem{Kong2} Q. Kong, H. Wu, A. Zettl, \textit{Dependence of the n-th Sturm-Liouville eigenvalue on the problem},
J. Differential Equations 156 (1999), 328--354.

\bibitem{Kong3}  Q. Kong, H. Wu, A. Zettl, \textit{Geometric aspects of Sturm-Liouville problems, I. Structures on spaces
of boundary conditions}, Proc. Roy. Soc. Edinburgh Sect. A  130 (2000), 561--589.

\bibitem{Kong4} Q. Kong, H. Wu, A. Zettl, \textit{Sturm-Liouville problems with finite spectrum}, J. Math. Anal. Appl.  263  (2001),  748--762.

\bibitem{Kong-Zettl1996} Q. Kong, A. Zettl, \textit{Dependence of eigenvalues of Sturm-Liouville problems on the boundary}, J. Differential Equations  126  (1996),  389--407.
\bibitem{Long2002}Y. Long, \textit{Index Theory for Symplectic Paths with Applications}, Progr.  Math.,  207, Birkh\"{a}user, Basel, 2002.

\bibitem{Rellich1} F. Rellich, \textit{St$\ddot{o}$rungstheorie der Spektralzerlegung},
Proceedings of the  International Congress of Mathematicians 1 (1950), 606--613.



\bibitem{Shi3} Y. Shi, S. Chen,  \textit{Spectral theory of second-order vector difference equations},
 J. Math. Anal. Appl. 239 (1999), 195--212.






\bibitem{Zettl1} A. Zettl, \textit{Sturm-Liouville Theory},
Mathematical Surveys Monographs. Amer. Math. Soc. 121 (2005).

\bibitem{Zhu5} H. Zhu,   \textit{A note on eigenvalues of a class of  singular continuous and discrete linear Hamiltonian systems}, arXiv: 1808.00397.

\bibitem{Zhu2}  H. Zhu,  Y. Shi, \textit{Continuous dependence of the n-th eigenvalue of self-adjoint discrete Sturm-Liouville problems on the problem}, J. Differential Equations 260 (2016), 5987--6016.
\bibitem{Zhu1} H. Zhu,  S. Sun,  Y. Shi, H. Wu, \textit{ Dependence of eigenvalues of certain closely discrete
Sturm-Liouville problems},
Complex Anal. Oper. Theory, 10 (2016), 667--702.
\end{thebibliography}
\end{document}